\documentclass[a4paper,twoside,11pt,oneside]{article}
\usepackage{verbatim}
\usepackage{graphicx}
\usepackage{listings}
\usepackage{amsmath,amssymb,amsthm}
\usepackage[utf8]{inputenc}
\usepackage[english]{babel}
\usepackage[top=4cm, bottom=4cm, left=2.5cm, right=2.5cm]{geometry}
\usepackage{fullpage}
\usepackage{colonequals}
\usepackage[colorlinks=true,linkcolor=black,urlcolor= black,citecolor=black]{hyperref}
\usepackage{graphicx}
\usepackage{color}
\usepackage{indentfirst}
\usepackage[OT2,T1]{fontenc}
\usepackage{comment} 
\usepackage{shadethm}
\usepackage{enumerate}
\usepackage{soul}
\usepackage{mathtools}
\usepackage{tikz}
\usetikzlibrary{matrix}
\usepackage[nottoc, notlof, notlot]{tocbibind}

\DeclareMathOperator{\et}{and}

\DeclareMathOperator{\If}{if}
\DeclareMathOperator{\surf}{surf}

\DeclareMathOperator{\dx}{dx}
\DeclareMathOperator{\dy}{dy}

\DeclareMathOperator{\loc}{loc}

\DeclareMathOperator{\dt}{dt}

\DeclareMathOperator{\Op}{Op}
\DeclareMathOperator{\sgn}{sgn}

\DeclareMathOperator{\simp}{simp}
\DeclareMathOperator{\WW}{WW}
\DeclareMathOperator{\BW}{BW}
\DeclareMathOperator{\Wh}{Wh}
\DeclareMathOperator{\inv}{inv}

\title{Rigorous derivation of the Whitham equations from the water waves equations in the shallow water regime}
\author{Louis Emerald}

\newtheorem{mydef}{Definition}[section]

\newtheorem{mypp}[mydef]{Proposition}
\newtheorem{myppr}[mydef]{Property}
\newtheorem{mythm}[mydef]{Theorem}
\newtheorem{mylem}[mydef]{Lemma}
\newtheorem{mycoro}[mydef]{Corollary}
\newtheorem{myrem}[mydef]{Remark}

\newtheorem{mynots}[mydef]{Notations}
\newtheorem{mynot}[mydef]{Notation}

\newtheorem{myhyp}[mydef]{Hypothesis}
\newtheorem{mydefpp}[mydef]{Definition/Property}

\usepackage{microtype}
\begin{document}
\numberwithin{equation}{section}

\maketitle
\begin{abstract}
    We derive the Whitham equations from the water waves equations in the shallow water regime using two different methods, thus obtaining a direct and rigorous link between these two models. The first one is based on the construction of approximate Riemann invariants for a Whitham-Boussinesq system and is adapted to unidirectional waves. The second one is based on a generalisation of Birkhoff's normal form algorithm for almost smooth Hamiltonians and is adapted to bidirectional propagation. In both cases we clarify the improved accuracy on the fully dispersive Whitham model with respect to the long wave Korteweg-de Vries approximation.  
\end{abstract}

\section{Introduction}
\subsection{Motivations}
In this paper we work on a unidirectionnal model for surface waves in coastal oceanography, the Whitham equation, introduced by Whitham in \cite{Whitham67} and \cite{Whitham74}. This equation, having the same dispersion relation as the general water waves model, is a full dispersion modification of the Korteweg-de Vries equation (KdV). As such one expects an improved accuracy and range of validity as a model for surface waves compared to the long wave KdV approximation. More precisely, we anticipate that in the presence of weak nonlinearities, the Whitham equation stays close to the water waves equations, even if the dispersion effects are not small, as obtained in \cite{Emerald2020} for the bidirectional full dispersion systems. 

The Whitham model has weaker dispersive effects compared to the KdV model and thus allows for wavebreaking and Stokes waves of maximal amplitude to occur, as one would expect of a model for surface waves in coastal oceanography. This was the historical reason for which Whitham introduced it. A rigorous proof of the existence of the wavebreaking phenomenon has been established by Hur \cite{Hur17} and Saut \& Wang \cite{SautWang20}. Ehrnström \& Wahlén \cite{EhrnstromWahlen19} and Truong \textit{et al.} \cite{TruongWahlenWheeler2020} proved the existence of Stokes waves of maximal amplitude. 

Besides \cite{Hur17,SautWang20,EhrnstromWahlen19, TruongWahlenWheeler2020} we emphasize some other works on the Whitham equation. Klein \textit{et al.} \cite{KleinLinaresPilodEtAl18} compared rigorously the solutions of the Whitham equation with those of the KdV equation and exhibited three different regimes of behaviour: scattering for small initial data, finite time blow-up and a KdV long wave regime. Ehrnström \& Wang \cite{EhrnstromWang} proved an enhanced existence time for solutions of the Whitham equation associated with small initial data. Ehrnström \& Kalisch \cite{EhrnstromKalisch2009} proved the existence of periodic traveling-waves solutions. Sanford \textit{et al.} \cite{SanfordKodamaCarterKalisch2014} established that large-amplitude periodic traveling-wave solutions are unstable, while those of small-amplitude are stable if their wavelength is long enough. Denoting by $\mu$ the shallow water parameter and by $\varepsilon$ the nonlinearity parameter, Klein \textit{et al.} \cite{KleinLinaresPilodEtAl18} proved the validity of the Whitham equation in the KdV regime $\{ \mu = \varepsilon \}$ by obtaining that the Whitham equation approximate the KdV equation at a precision of order $\varepsilon^2$. Moldabayev \textit{et al.} \cite{DutykhKalischMoldabayev2015} derived the Whitham equation from the Hamiltonian of the water waves equations in the unusual regime $\{ \varepsilon = e^{-\frac{\kappa}{\mu^{\nu/2}}}\}$, where $\kappa$ and $\nu$ are parameters numerically inferred from the solitary waves of the Whitham model. They also assumed the initial data of the water waves equations prepared to generate unidirectional waves. If we transpose their derivation in the general shallow water regime $\{ 0 \leq \mu \leq \mu_{\max}, 0 \leq \varepsilon \leq 1 \}$, where $\mu_{\max} > 0$ is fixed, we obtain a precision of order $\mu \varepsilon + \varepsilon^2$. To our point of view, the two previous results are not satisfying as the authors did not characterize the range of validity of the Whitham equation as a model for water waves with enough accuracy. In this work we propose to improve these results using two different methods.

The first method is based on an adaptation of the one used to derive the inviscid Burgers equations from the Nonlinear Shallow Water system using the Riemann invariants of the latter and requires, as in \cite{DutykhKalischMoldabayev2015}, that initial data are prepared to generate unidirectional waves. We show from it that the Whitham equation can be derived from the water waves equations at the order of precision $\mu\varepsilon$, which is the same order as for the Whitham-Boussinesq equations when considering general initial data (see \cite{Emerald2020} for a rigorous derivation of the latter equations in the shallow water regime). From this derivation, we prove that solutions of the water waves equations, associated with well-prepared initial data, can be approximated in the shallow water regime, to within order $\mu \varepsilon t$ in a time scale of order $\varepsilon^{-1}$, by a right or left propagating wave solving a Whitham-type equation.

The second method does not need well-prepared initial data. It is based on a generalization of Birkhoff's normal form algorithm for almost smooth functions (see Definition \ref{almost smooth functions}) developed by Bambusi in \cite{Bambusi20} to show that the solutions of the water waves equations can be approximated in the KdV regime, to within order $\varepsilon$ in a time scale of order $\varepsilon^{-1}$, by two counter-propagating waves solving KdV-type equations, thus improving both the results obtained by Schneider and Wayne in \cite{SchneiderWayne2012}, and those obtained by Bona, Colin and Lannes in \cite{BonaColinLannes2005}. If we transpose Bambusi's result in the shallow water regime, we get a precision of order $\varepsilon^2 + (\mu^2 + \mu\varepsilon + \varepsilon^2)t$ in a time scale of order $(\max{(\mu,\varepsilon)})^{-1}$, which tells us that the two KdV-type equations are not appropriate to approximate the water waves equations when $\mu$ is not small. Bambusi's method uses the local structure of the KdV equation. We adapt it to the nonlocal Whitham equation and prove that the solutions of the water waves equations can be approximated in the shallow water regime, to within order $\varepsilon^2 + (\mu\varepsilon  + \varepsilon^2) t$ in a time scale of order $(\max{(\mu,\varepsilon)})^{-1}$, by two counter-propagating waves solving Whitham-type equations. In addition, we express explicitly the loss of regularity needed to derive the two uncoupled Whitham-type equations solved by the counter-propagating waves. 

We now compare the results from the first method and Bambusi's method for the KdV and the Whitham equations. In the presence of strong nonlinearities ($\varepsilon \geq \mu$), we get from the first method a precision order $\mu$ in a time scale of order $\varepsilon^{-1}$, whereas Bambusi's method for the KdV equations and for the Whitham equations gives a precision of order $\mu + \varepsilon$ in the same time scale. The $\varepsilon$ in the latter order of precision comes from the coupling effects between the two counter-propagating waves. Hence, if $\mu$ is small, we have a good approximation from the Whitham equation, and not from the KdV equation, of the solutions of the water waves equations associated with well-prepared initial data, even if $\varepsilon$ is not small. In the presence of weak nonlinearities ($\varepsilon \leq \mu$), we get from the first method and Bambusi's method for the Whitham equations a precision of order $\varepsilon$ in a time scale of order $\mu^{-1}$, whereas Bambusi's method for the KdV equations gives a precision of order $\mu + \varepsilon$ in the same time scale. Hence, if $\varepsilon$ is small, we have a good approximation from the Whitham equation, and not from the KdV equation, of the solutions of the water waves equations, even if $\mu$ is not small. 

\subsection{Main results}

The starting point of our study is the water waves equations:
\begin{align}\label{WaterWavesEquations}
    \begin{cases}
        \partial_t \zeta - \frac{1}{\mu}\mathcal{G}^{\mu}[\varepsilon\zeta]\psi = 0, \\
        \partial_t \psi + \zeta + \frac{\varepsilon}{2}(\partial_x\psi)^2-\frac{\mu\varepsilon}{2}\frac{(\frac{1}{\mu}\mathcal{G}^{\mu}[\varepsilon\zeta]\psi + \varepsilon\partial_x\zeta\cdot\partial_x\psi)^2}{1+\varepsilon^2\mu(\partial_x\zeta)^2} = 0.
    \end{cases}
\end{align}

Here 
\begin{itemize}
    \item The free surface elevation is the graph of $\zeta$, which is a function of time $t$ and horizontal space $x\in\mathbb{R}$.
    \item $\psi(t,x)$ is the trace at the surface of the velocity potential. 
    \item $\mathcal{G}^{\mu}$ is the Dirichlet-Neumann operator defined later in Definition \ref{Dirichlet to Neumann operator}.
\end{itemize}
Every variable and function in \eqref{Hamiltonian Water Waves equations} is compared with physical characteristic parameters of the same dimension. Among those are the characteristic water depth $H_0$, the characteristic wave amplitude $a_{\surf}$ and the characteristic wavelength $L$. From these comparisons appear two adimensional parameters of main importance:
\begin{itemize}
    \item $\mu \colonequals \frac{H_0^{2}}{L^{2}}$: the shallow water parameter,
    \item $\varepsilon \colonequals \frac{a_{\surf}}{H_0}$: the nonlinearity parameter.
\end{itemize}
We refer to \cite{WWP} for details on the derivation of these equations.
\indent 

In \cite{Zakharov68}, Zakharov proved that the water waves system \eqref{WaterWavesEquations} enjoys an Hamiltonian formulation. Let $H_{\WW}$ be defined
by 
\begin{align}\label{Hamiltonian Water Waves equations}
    H_{\WW} \colonequals \frac{1}{2} \int_{\mathbb{R}} \zeta^2 \dx + \frac{1}{2} \int_{\mathbb{R}} \psi \frac{1}{\mu} \mathcal{G}^{\mu}[\varepsilon\zeta]\psi \dx.
\end{align}
Then \eqref{WaterWavesEquations} is equivalent to the Hamilton equations
\begin{align*}
    \begin{cases}
        \partial_t \zeta = \delta_{\psi} H_{\WW},\\
        \partial_t \psi = -\delta_{\zeta} H_{\WW},
    \end{cases}
\end{align*}
where $\delta_{\zeta}$ and $\delta_{\psi}$ are functional derivatives.

\indent 

Before giving the main definitions of this section, here is one assumption maintained throughout this paper.

\begin{myhyp}
    A fundamental hypothesis is the lower boundedness by a positive constant of the water depth (non-cavitation assumption):
    \begin{align}\label{Non-Cavitation Hypothesis}
        \exists h_{\min}>0, \forall x \in \mathbb{R}, \ \ h \colonequals 1 + \varepsilon \zeta(t,x) \geq h_{\min}.
    \end{align}
\end{myhyp}

Moreover, we will always work in the shallow water regime which we define here.
\begin{mydef}\label{shallow water regime}
    Let $\mu_{\max} > 0$, we define the shallow water regime $\mathcal{A} \colonequals \{(\mu,\varepsilon), \ \ 0\leq \mu\leq \mu_{\max}, \ \ 0 \leq \varepsilon \leq 1\}$.
\end{mydef}

In what follows we need some notations on the functional framework of this paper.
\begin{mynots}\label{functional framework}
    \begin{itemize}
        \item For any $\alpha \geq 0$ we will respectively denote $H^{\alpha}(\mathbb{R}), W^{\alpha,1}(\mathbb{R})$ and $W^{\alpha,\infty}(\mathbb{R})$ the Sobolev spaces of order $\alpha$ in respectively $L^2(\mathbb{R}), L^1(\mathbb{R})$ and $L^{\infty}(\mathbb{R})$. We will denote their associated norms by $|\cdot|_{H^{\alpha}}, |\cdot|_{W^{\alpha,1}}$ and $|\cdot|_{W^{\alpha,\infty}}$. 
        \item For any $\alpha \geq 0$ we will denote $\dot{H}^{\alpha+1}(\mathbb{R}) \colonequals \{ f \in L_{\loc}^{2}(\mathbb{R}), \ \ \partial_x f \in H^{\alpha}(\mathbb{R}) \}$ and $\dot{W}^{\alpha+1,1}(\mathbb{R}) \colonequals \{ f \in L_{\loc}^{1}(\mathbb{R}), \ \ \partial_x f \in W^{\alpha,1}(\mathbb{R}) \}$ the Beppo-Levi spaces of order $\alpha$. Their associated seminorms are respectively $|\cdot|_{\dot{H}^{\alpha+1}} \colonequals |\partial_x(\cdot)|_{H^{\alpha}}$ and $|\cdot|_{\dot{W}^{\alpha+1,1}} \colonequals |\partial_x(\cdot)|_{W^{\alpha,1}}$. 
    \end{itemize}
\end{mynots}

Now we define the Dirichlet-Neumann operator.
\begin{mydef}\label{Dirichlet to Neumann operator}
    For all $\zeta$ and $\psi$ sufficiently smooth, let $\phi$ be the velocity potential solving the elliptic problem
    \begin{align*}
        \begin{cases}
            (\mu \partial_x^2 + \partial_z^2) \phi = 0, \\
            \phi|_{z=\varepsilon \zeta} = \psi, \ \ \partial_z\phi|_{z=-1} = 0. 
        \end{cases}
    \end{align*}
    where $z \in (-1,\varepsilon \zeta)$ is the vertical space variable.
    
    We define the Dirichlet-Neumann operator by the formula
    \begin{align*}
        \mathcal{G}^{\mu}[\varepsilon\zeta]\psi = \sqrt{1 + \varepsilon^2 (\partial_x\zeta)^2}\partial_{\textbf{n}^{\mu}}\phi|_{\varepsilon\zeta},
    \end{align*}
    where $\textbf{n}^{\mu}$ is the outward normal vector of the free surface $\varepsilon \zeta$. It depends on time and space.
    
    $\mathcal{G}^{\mu}$ is linear in $\psi$ and nonlinear in $\zeta$. See \cite{WWP} for a thorough study of this operator.
\end{mydef}

We also recall the definition of a Fourier multiplier.
\begin{mydef}
    Let $u:\mathbb{R} \to \mathbb{R}$ be a tempered distribution, let $\widehat{u}$ be its Fourier transform. Let $F \in \mathcal{C}^{\infty}(\mathbb{R})$ be a smooth function such that there exists $m \in \mathbb{Z}$ for which $\forall \beta \geq 0$, $|\partial_{\xi}^{\beta} F(\xi)| \leq \widetilde{C}_{\beta} (1+ |\xi|)^{m-|\beta|}$, where $\widetilde{C}_{\beta} > 0$ is a constant depending on $\beta$. Then the Fourier multiplier associated with $F$ is denoted $\mathrm{F}(D)$ (denoted $\mathrm{F}$ when no confusion is possible) and defined by the formula:
    \begin{align*}
        \widehat{\mathrm{F}(D)u}(\xi) = F(\xi)\widehat{u}(\xi).
    \end{align*}
\end{mydef}

The first method to rigorously justify the Whitham equations from the water waves equations gives the two following results.

\begin{mypp}\label{Diagonalization of the WB equations}
    Let $\mu_{\max} > 0$. Let also $\mathrm{F}_{\mu} = \sqrt{\frac{\tanh{(\sqrt{\mu}|D|)}}{\sqrt{\mu}|D|}}$ be a Fourier multiplier. There exists $n \in \mathbb{N}^*$ such that for any $(\mu,\varepsilon) \in \mathcal{A}$ and $\zeta \in L^{\infty}_t H^{\alpha+n}_x(\mathbb{R})$ and $\psi \in L^{\infty}_t\dot{H}^{\alpha+n}_x(\mathbb{R})$ solutions of the water waves equations \eqref{WaterWavesEquations} satisfying \eqref{Non-Cavitation Hypothesis}, there exists $R_1, R_2 \in H^{\alpha}(\mathbb{R})$ uniformly bounded in $(\mu,\varepsilon)$ such that the quantities
    \begin{align}\label{Diagonalized system introduction}
        u^+ = \frac{\sqrt{h}-1}{\varepsilon} + \frac{1}{2}\mathrm{F}_{\mu}^{-1}[v], \ \ u^- = -\frac{\sqrt{h}-1}{\varepsilon} + \frac{1}{2}\mathrm{F}_{\mu}^{-1}[v],
    \end{align}
    where $v = \mathrm{F}_{\mu}^2[\partial_x \psi]$, satisfy the equations
    \begin{align*}
        \begin{cases}
            \partial_t u^+ + (\varepsilon\frac{3 u^+ + u^-}{2}+1) \mathrm{F}_{\mu}\partial_x u^+ = \mu\varepsilon R_1, \\
            \partial_t u^- + (\varepsilon\frac{u^+ + 3u^-}{2}-1) \mathrm{F}_{\mu}\partial_x u^- = \mu\varepsilon R_2.
        \end{cases}
    \end{align*}
\end{mypp}

\begin{myrem}\label{Remark Riemann invariants}
    The quantities $u^+$ and $u^-$ can almost be seen as Riemann invariants of the Whitham-Boussinesq equations
    \begin{align}\label{Whitham-Boussinesq system introduction}
    \begin{cases}
        \partial_t \zeta + \partial_x^2 \psi  + \varepsilon \partial_x(\zeta \partial_x \psi)  + \varepsilon \zeta \partial_x^2 \psi  = 0, \\
        \partial_t \mathrm{F}_{\mu}^2[\partial_x \psi] + \mathrm{F}_{\mu}^2[\partial_x \zeta] + \varepsilon \partial_x\psi \partial_x^2 \psi = 0,
    \end{cases}
\end{align}
for which we know the rigorous derivation from the water waves equations (see \cite{Emerald2020}). This is the main idea to prove this result.
\end{myrem}

Using Theorem 4.16 and Theorem 4.18 of \cite{WWP} (the Rayleigh-Taylor condition is always satisfied when the bottom is flat, see Subsection 4.3.5 in \cite{WWP}) we get the following corollary.

\begin{mycoro}\label{First corollary introduction}
    For any $\alpha \geq 0$, there exists $n \in \mathbb{N}^*$ such that for any $(\mu,\varepsilon) \in \mathcal{A}$, we have what follows. Consider the Cauchy problem for the water wave problem \eqref{WaterWavesEquations} with initial conditions $(\zeta_0,\psi_0) \in H^{\alpha+n}(\mathbb{R})\times \dot{H}^{\alpha+n+1}(\mathbb{R})$ satisfying the non-cavitation assumption \eqref{Non-Cavitation Hypothesis}, 
    and denote by $(\zeta, \psi)$ the corresponding solution. Let $(u^+_{e,0},u^-_{e,0})$ be defined by the formulas \eqref{Diagonalized system introduction} applied to $(\zeta_0,\psi_0)$. There exists a unique solution $(u^+_e,u^-_e)$ of the exact diagonalized system 
    \begin{align}\label{exact diagoanlized system introduction}
        \begin{cases}
            \partial_t u^+_e + (\varepsilon\frac{3 u^+_e + u^-_e}{2}+1) \mathrm{F}_{\mu}\partial_x u^+_e = 0, \\
            \partial_t u^-_e + (\varepsilon\frac{u^+_e + 3u^-_e}{2}-1) \mathrm{F}_{\mu}\partial_x u^-_e = 0,
        \end{cases}
    \end{align}
    with initial conditions $(u^+_{e,0},u^-_{e,0})$, which satisfy the following property: denote by $(\zeta_{\rm{c}}, \psi_{\rm{c}})$ the following quantities
    \begin{align}\label{zetac and psic}
        \zeta_{\rm{c}} = \frac{1}{\varepsilon}\Big[ (\frac{\varepsilon}{2}(u^+_e - u^-_e) + 1)^2 - 1 \Big] \ \ \et \ \ \psi_{\rm{c}} = \int_0^x \mathrm{F}_{\mu}^{-1}[u^+_e + u^-_e] \dx,
    \end{align}
    then for all times $t \in [0,\frac{T}{\varepsilon}]$, one has 
    \begin{align*}
        |(\zeta-\zeta_{\rm{c}}, \psi -\psi_{\rm{c}})|_{H^{\alpha}\times \dot{H}^{\alpha+1}} \leq \mu \varepsilon C t,
    \end{align*}
    where $C, T^{-1} = C(\frac{1}{h_{\min}}, \mu_{\max}, |\zeta_0|_{H^{\alpha+n}}, |\psi_0|_{\dot{H}^{\alpha+n+1}})$. 
\end{mycoro}

\begin{myrem}
    The existence and uniqueness of the solution $(u_e^+,u_e^-)$ of the exact diagonalized system \eqref{exact diagoanlized system introduction} with initial conditions $(u^+_{e,0},u^-_{e,0})$ on the required time scale is given by Proposition \ref{local existence}.
\end{myrem}

To write the next result we need another notation.
\begin{mynot}\label{Big O Sobolev spaces}
    Let $k\in\mathbb{N}$ and $ l\in\mathbb{N}$. A function $R$ is said to be of order $O(\mu^k\varepsilon^l)$, denoted $R=O(\mu^k\varepsilon^l),$ if divided by $\mu^k\varepsilon^l$ this function is uniformly bounded with respect to $(\mu,\varepsilon) \in \mathcal{A}$ in the Sobolev norms $|\cdot|_{H^{\alpha}}$, $\alpha \geq 0$.
\end{mynot}

\begin{mypp}\label{First main result}
    With the same hypotheses and notations as in Proposition \ref{Diagonalization of the WB equations}, if $u^-(0) = O(\mu)$, then there exists $T_1 > 0$ such that for all times $t \in [0, \frac{T_1}{\epsilon}]$, $u^-(t) = O(\mu)$. Moreover for these times, $u^+$ satisfies the Whitham equation up to a remainder term of order $O(\mu \varepsilon)$, i.e.
    \begin{align*}
        \partial_t u^+ + \mathrm{F}_{\mu}[\partial_x u^+] + \frac{3\varepsilon}{2}u^+\partial_x u^+ = O(\mu\epsilon).
    \end{align*}
    If instead $u^+(0) = O(\mu)$, then there exists $T_2 > 0$ such that for all times $t \in [0, \frac{T_2}{\epsilon}]$, $u^+(t) = O(\mu)$. Moreover for these times, $u^-$ satisfies the counter propagating Whitham equation up to remainder term of order $O(\mu\varepsilon)$, i.e.
    \begin{align}\label{left front Whitham equation}
        \partial_t u^- - \mathrm{F}_{\mu}[\partial_x u^-] + \frac{3\varepsilon}{2}u^-\partial_x u^- = O(\mu\varepsilon).
    \end{align}
\end{mypp}

\begin{mypp}\label{Corollary convergence introduction}
    In Corollary \ref{First corollary introduction}, if $u^-_e(0) = O(\mu)$, then one can replace $\zeta_{\rm{c}}$ and $\psi_{\rm{c}}$ by 
    \begin{align*}
        \zeta_{\Wh,+} = \frac{1}{\varepsilon}\Big[ (\frac{\varepsilon}{2}u^+ + 1)^2 - 1 \Big] \ \ \et \ \ \psi_{\Wh,+} = \int_0^x \mathrm{F}_{\mu}^{-1}[u^+] \dx,
    \end{align*}
    where $u^+ \in \mathcal{C}([0,\frac{T}{\varepsilon}],H^{\alpha}(\mathbb{R}))$ solves the exact Whitham equation
    \begin{align*}
        \partial_t u^+ + \mathrm{F}_{\mu}[\partial_x u^+] + \frac{3\varepsilon}{2}u^+\partial_x u^+ = 0,
    \end{align*}
    and get for all times $t \in [0,\frac{T}{\varepsilon}]$,
    \begin{align*}
        |(\zeta-\zeta_{\Wh,+}, \psi -\psi_{\Wh,+})|_{H^{\alpha}\times \dot{H}^{\alpha+1}} \leq C (|u_e^-(0)|_{H^{\alpha+1}} + \mu\varepsilon t)  ,
    \end{align*}
    where $C,T^{-1} = C(\frac{1}{h_{\min}}, \mu_{\max}, |\zeta_0|_{H^{\alpha+n}}, |\psi_0|_{\dot{H}^{\alpha+n+1}})$. 
    
    If instead $u^+_e(0) = O(\mu)$, then one can replace $\zeta_{\rm{c}}$ and $\psi_{\rm{c}}$ by 
    \begin{align*}
        \zeta_{\Wh,-} = \frac{1}{\varepsilon}\Big[ (\frac{\varepsilon}{2}u^- - 1)^2 - 1 \Big] \ \ \et \ \ \psi_{\Wh,-} = \int_0^x \mathrm{F}_{\mu}^{-1}[u^-] \dx,
    \end{align*}
    where $u^- \in \mathcal{C}([0,\frac{T}{\varepsilon}],H^{\alpha}(\mathbb{R}))$ solves the exact counter-propagating Whitham equation 
    \begin{align*}
        \partial_t u^- - \mathrm{F}_{\mu}[\partial_x u^-] + \frac{3\varepsilon}{2}u^-\partial_x u^- = 0,
    \end{align*}
    and get for all times $t \in [0,\frac{T}{\varepsilon}]$,
    \begin{align*}
        |(\zeta-\zeta_{\Wh,-}, \psi -\psi_{\Wh,-})|_{H^{\alpha}\times \dot{H}^{\alpha+1}} \leq C (|u_e^+(0)|_{H^{\alpha+1}} + \mu \varepsilon t),
    \end{align*}
    where $C,T^{-1} = C(\frac{1}{h_{\min}}, \mu_{\max}, |\zeta_0|_{H^{\alpha+n}}, |\psi_0|_{\dot{H}^{\alpha+n+1}})$. 
\end{mypp}

The second method is based on a generalisation of Birkhoff's normal form algorithm developed by Bambusi in \cite{Bambusi20}. Its application to our case leads us to define some operator and transformations which we need to write our main results.

\begin{mydef}\label{Definition transformations}
    \begin{itemize}
        \item Let $\alpha \geq 0$. We define $\mathcal{T}_{\rm{I}}: H^{\alpha}(\mathbb{R})\times H^{\alpha}(\mathbb{R}) \to H^{\alpha}(\mathbb{R})\times \dot{H}^{\alpha+1}(\mathbb{R})$ by the formula
        \begin{align*}
            \mathcal{T}_{\rm{I}}(\zeta,v) = \begin{pmatrix} \zeta \\ \int_0^x v(y) \dy  \end{pmatrix}.
        \end{align*}
        \item Let $\alpha \geq 0$. We define $\mathcal{T}_{\rm{D}}: H^{\alpha+1}(\mathbb{R})\times H^{\alpha+1}(\mathbb{R}) \to H^{\alpha}(\mathbb{R}) \times H^{\alpha}(\mathbb{R})$ by the formula 
            \begin{align*}
                \mathcal{T}_{\rm{D}}(r,s) = \begin{pmatrix} r+s \\ \mathrm{F}_{\mu}^{-1}[r-s] \end{pmatrix}.
            \end{align*}
        \item We define $\partial^{-1}: L^1(\mathbb{R}) \to L^{\infty}(\mathbb{R})$ by the formula
            \begin{align*}
                \partial^{-1} u(y) = \frac{1}{2} \int_{\mathbb{R}} \sgn(y-y_1) u(y_1) \dy_1.
            \end{align*}
        \item Let $\alpha \geq 0$ be a positive integer. We define $ \mathcal{T}_{\rm{B}}: W^{\alpha+1,1}(\mathbb{R})\times W^{\alpha+1,1}(\mathbb{R}) \to W^{\alpha,1}(\mathbb{R})\times W^{\alpha,1}(\mathbb{R})$ by the formula
            \begin{align*}
                \mathcal{T}_{\rm{B}}(r,s) = \begin{pmatrix} r + \frac{\varepsilon}{4} \partial_x (r) \partial^{-1}(s) + \frac{\varepsilon}{4} rs + \frac{\varepsilon}{8}s^2 \\ s + \frac{\varepsilon}{4} \partial_x (s) \partial^{-1}(r) + \frac{\varepsilon}{4} rs + \frac{\varepsilon}{8}r^2 \end{pmatrix}.
            \end{align*} 
    \end{itemize}
\end{mydef}

\begin{mythm}\label{Final theorem introduction}
    Let $\alpha \geq 0$ be a positive integer and $(\mu,\varepsilon) \in \mathcal{A}$. Let $r,s$ $\in \mathcal{C}^1([0,\frac{T}{\varepsilon}], W^{\alpha+7,1}(\mathbb{R}))$ be solutions of the Whitham equations
    \begin{align}\label{Decoupled Whitham equations introduction}
        \begin{cases}
            \partial_t r + \mathrm{F}_{\mu}[\partial_x r] + \frac{3\varepsilon}{2} r\partial_x r = 0, \\
            \partial_t s -\mathrm{F}_{\mu}[\partial_x s] - \frac{3\varepsilon}{2} s \partial_x s = 0.
        \end{cases}
    \end{align}
    Let also $H_{\WW}$ be the Hamiltonian of the water waves equations \eqref{Hamiltonian Water Waves equations} and $J = \begin{pmatrix} 0 & 1 \\ -1 & 0 \end{pmatrix}$ the Poisson tensor associated with. Then the quantities 
    \begin{align}\label{zeta and partial psi introduction}
        \begin{pmatrix} \zeta_{\Wh} \\ \psi_{\Wh} \end{pmatrix} \colonequals \mathcal{T}_{\rm{I}}(\mathcal{T}_{\rm{D}}(\mathcal{T}_{\rm{B}}(r,s)),
    \end{align}
    satisfy
    \begin{align*}
        \partial_t \begin{pmatrix} \zeta_{\Wh} \\ \psi_{\Wh} \end{pmatrix} = J \nabla (H_{\WW}) (\zeta_{\Wh},\psi_{\Wh}) + (\mu\varepsilon + \varepsilon^2) R, \ \ \forall t \in [0,\frac{T}{\varepsilon}], 
    \end{align*}
    where $|R|_{H^{\alpha}\times \dot{H}^{\alpha+1}} \leq C(\mu_{\max},|r|_{W^{\alpha+7,1}},|s|_{W^{\alpha+7,1}})$.
\end{mythm}

\begin{myrem}
    \begin{itemize}
        \item The existence of $r,s \in \mathcal{C}^1([0,\frac{T}{\varepsilon}], W^{\alpha+6,1}(\mathbb{R}))$ solutions of system \eqref{Decoupled Whitham equations introduction} is given by Lemma \ref{Solutions of the Whitham equations in W alpha,1}.
        \item The remainder $R$ is uniformly bounded in $(\mu,\varepsilon)$ for times $t \in [0,\frac{T}{\max{(\mu,\varepsilon)}}]$ in $H^{\alpha}(\mathbb{R})\times \dot{H}^{\alpha+1}(\mathbb{R})$. The particular time interval $[0,\frac{T}{\max{(\mu,\varepsilon)}}]$ comes from the estimates of $r$ and $s$ in $W^{\alpha+7,1}(\mathbb{R})$ given by Lemma \ref{Solutions of the Whitham equations in W alpha,1}. Indeed, due to the dispersive effects one has a loss of decrease of order $O((1+\mu t)^{\theta})$ for any $\theta >1/2$, in the Sobolev spaces $W^{\beta,1}(\mathbb{R})$ where $\beta \geq 0$.
        \item The second equation of \eqref{Decoupled Whitham equations introduction} is different from \eqref{left front Whitham equation}. One has to replace the definition of $u^-$ in Proposition \ref{Diagonalization of the WB equations} by $u^- = \frac{\sqrt{h}-1}{\varepsilon} - \frac{1}{2}\mathrm{F}_{\mu}^{-1}[v]$ to get 
        \begin{align*}
            \partial_t u^- - \mathrm{F}_{\mu}[\partial_x u^-] - \frac{3\varepsilon}{2}u^-\partial_x u^- = \mu\varepsilon R,
        \end{align*}
        when setting $u^+(0) = O(\mu)$.
        \item In Theorem \ref{Final theorem introduction} we express the loss of regularity needed to pass rigorously from the Whitham equations \eqref{Decoupled Whitham equations introduction} to the water waves equations \eqref{WaterWavesEquations}. 
    \end{itemize}
\end{myrem}

Using again Theorem 4.16 and Theorem 4.18 of \cite{WWP}, we get the following corollary.

\begin{mycoro}\label{Final corollary}
    There exists $n \in \mathbb{N}^*$ such that for any $\alpha \geq 0$ and any $(\mu,\varepsilon) \in \mathcal{A}$, we have what follows. Consider the Cauchy problem for the water wave problem \eqref{WaterWavesEquations} with initial conditions $(\zeta_0,\psi_0) \in W^{\alpha+n,1}(\mathbb{R})\times \dot{W}^{\alpha+n+1,1}(\mathbb{R})$ satisfying the non-cavitation assumption \eqref{Non-Cavitation Hypothesis}, 
    and denote by $(\zeta, \psi)$ the corresponding solution.
    There exists a solution $(r,s)$ of the Hamiltonian system \eqref{Decoupled Whitham equations introduction} which satisfies the following property: denote by $(\zeta_{\Wh},\psi_{\Wh})$ the solution defined by \eqref{zeta and partial psi introduction}, then for all times $t \in [0,\frac{T}{\max{(\mu,\varepsilon)}}]$, one has
    \begin{align*}
        |(\zeta-\zeta_{\Wh},\psi-\psi_{\Wh})|_{H^{\alpha}\times \dot{H}^{\alpha+1}} \leq C(\varepsilon^2+(\mu\varepsilon + \varepsilon^2)t),
    \end{align*}
    where $C,T^{-1} = C(\frac{1}{h_{\min}},\mu_{\max},|\zeta_0|_{W^{\alpha+n}},|\psi_0|_{\dot{W}^{\alpha+n+1}})$. See \eqref{r0 and s0} for the initial conditions $(r_0, s_0)$ associated with $(r,s)$.
\end{mycoro}

\subsection{Outline}
In section \ref{Derivation of the Whitham equations from the Riemann invariants of a Whitham-Boussinesq system} we prove Proposition \ref{Diagonalization of the WB equations} and Proposition \ref{First main result}. In subsection \ref{Formal diagonalization of the Whitham-Boussinesq equation} we focus on the proof of Proposition \ref{Diagonalization of the WB equations} using symbolic calculus to diagonalize system \eqref{Whitham-Boussinesq system introduction} (see remark \ref{Remark Riemann invariants}). In subsection \ref{From the diagonalized Whitham-Boussinesq system to the Whitham equations} we prove Proposition \ref{First main result} using a local well-posedness result (see Proposition \ref{local existence}) and a stability result (see Proposition \ref{Stability}) on the diagonalized system \eqref{exact diagoanlized system introduction}. Then we prove Proposition \ref{Corollary convergence introduction}.

The section \ref{Decoupling the water waves equations into two Whitham equations} is dedicated to the proof of Theorem \ref{Final theorem introduction}. In subsection \ref{From the water waves equations to a Whitham-Boussinesq system} we do a first approximation of the water waves' Hamiltonian by the one of a specific Whitham-Boussinesq similar to \eqref{Whitham-Boussinesq system introduction} (see Proposition \ref{From HBX to HWW}). Then we do a simple change of unknowns separating the waves into a right and left front (see Proposition \ref{Transformation Diag}). In subsection \ref{Application of Birkhoff's algorithm} we apply the generalised Birkhoff's algorithm for almost smooth functions to the Hamiltonian obtained in the previous subsection. It gives an explicit transformation \eqref{explicit formulation of T} allowing to pass from the latter Hamiltonian to the one of a system composed of two decoupled Whitham equations (see Property \ref{property hamilton's equations HWh} and remark \ref{Remark two decoupled Whitham equations}) at the desired order of precision. In subsection \ref{From the Hamiltonian of the Whitham-Boussinesq system under normal form to two decoupled Whitham equations} we prove the existence of solutions of the two decoupled Whitham equations in the suitable Sobolev spaces $W^{\alpha,1}(\mathbb{R})$ for every integer $\alpha \geq 0$ (see lemma \ref{Solutions of the Whitham equations in W alpha,1}). Then we prove an intermediate theorem (see Theorem \ref{From HWh to HBW}) which describes the action of the transformation resulting from the generalised Birkhoff's algorithm on the two decoupled Whitham equations. At the end, we prove Theorem \ref{Final theorem introduction} and corollary \ref{Final corollary}. 

\section{Derivation of the Whitham equations from the Riemann invariants of a Whitham-Boussinesq system}
\label{Derivation of the Whitham equations from the Riemann invariants of a Whitham-Boussinesq system}
The goal of this section is to prove the Propositions \ref{Diagonalization of the WB equations} and \ref{First main result}.

We consider the following Whitham-Boussinesq equations. 
\begin{align}\label{Whitham-Boussinesq system}
    \begin{cases}
        \partial_t \zeta + \partial_x v + \varepsilon \partial_x(\zeta) v + \varepsilon \zeta \partial_x v = 0, \\
        \partial_t v + \mathrm{F}_{\mu}^2[\partial_x \zeta] + \varepsilon v \partial_x v = 0,
    \end{cases}
\end{align}
where $v = \mathrm{F}_{\mu}^2[\partial_x \psi] = \frac{\tanh{(\sqrt{\mu}|D|)}}{\sqrt{\mu}|D|}[\partial_x \psi]$.
Using Proposition 1.15 of \cite{Emerald2020} we easily get the following proposition.

\begin{mypp}\label{Consistency of the WB equations}
Let $\mu_{\max} > 0$. There exists $n \in \mathbb{N}^*$ and $T > 0$ such that for all $\alpha \geq 0$ and $p = (\mu,\varepsilon) \in \mathcal{A}$ (see Definition \ref{shallow water regime}), and for every solution $(\zeta,\psi) \in C([0,\frac{T}{\varepsilon}];H^{\alpha+n}(\mathbb{R})\times \dot{H}^{\alpha+n+1}(\mathbb{R}))$ to the water waves equations \eqref{WaterWavesEquations} one has 
\begin{align}
    \begin{cases}
        \partial_t \zeta + \partial_x v + \varepsilon \partial_x(\zeta v) = \mu \varepsilon R_1, \\
        \partial_t v + \mathrm{F}_{\mu}^2[\partial_x \zeta] + \varepsilon v \partial_x v = \mu\varepsilon R_2,
    \end{cases}
\end{align}
with $|R_1|_{H^{\alpha}},|R_2|_{H^{\alpha}} \leq C(\frac{1}{h_{\min}}, \mu_{\max}, |\zeta|_{H^{\alpha+n}},|\psi|_{\dot{H}^{\alpha+n+1}})$.

We say that the water waves equations are consistent with the Whitham-Boussinesq equations \eqref{Whitham-Boussinesq system} at precision order $O(\mu\varepsilon)$. 
\end{mypp}

\begin{myrem}
    We only chose a Whitham-Boussinesq system fitted for the computations. There exists a whole class of these systems. For example, one can add a regularizing Fourier multiplier such as $\mathrm{F}_{\mu}^2$ on the nonlinear terms of system \eqref{Whitham-Boussinesq system} and still be precise at order $O(\mu\varepsilon)$. The system thus obtained is 
    \begin{align}\label{WB system Dinvay}
        \begin{cases}
            \partial_t \zeta + \partial_x v + \varepsilon \mathrm{F}_{\mu}^2\partial_x[\zeta v] = 0, \\
            \partial_t v + \mathrm{F}_{\mu}^2[\partial_x \zeta] + \frac{\varepsilon}{2}\mathrm{F}_{\mu}^2[v \partial_x v] = 0.
    \end{cases}
    \end{align}
    In \cite{DinvaySelbergTesfahun19} the authors proved the local well-posedness of \eqref{WB system Dinvay}. The reasoning of this section is the same for any Whitham-Boussinesq system such as the latter.
\end{myrem}

The system \eqref{Whitham-Boussinesq system} is the starting point of our reasoning. We will adapt the method used to derive the inviscid Burgers equations from the Nonlinear Shallow Water equations which use the Riemann invariants of the latter. 

\indent 

Throughout this section we will use Notation \ref{Big O Sobolev spaces}.

\subsection{Diagonalization of the Whitham-Boussinesq equations}
In this subsection we formally diagonalize the system \eqref{Whitham-Boussinesq system} to get Proposition \ref{Diagonalization of the WB equations}. Because of the consistency of the water waves equations with the system \eqref{Whitham-Boussinesq system}, we can take $\zeta$ and $v$ solutions of the latter instead of taking solutions of the water waves equations. For our purpose, we just need to prove the following proposition. 

\begin{mypp}\label{Diag WB section 2}
    Let $\mu_{\max} > 0$, there exists $n \in \mathbb{N}$ such that for any $\alpha \geq 0$ and $(\zeta,v) \in \mathcal{C}([0,\frac{T}{\varepsilon}];H^{\alpha+n}(\mathbb{R})\times H^{\alpha+n}(\mathbb{R}))$ solutions of system \eqref{Whitham-Boussinesq system} satisfying the non-cavitation hypothesis \eqref{Non-Cavitation Hypothesis}, the quantities 
    \begin{align}\label{Riemann invariants}
        u^+ = 2\frac{\sqrt{h}-1}{\varepsilon} + \mathrm{F}_{\mu}^{-1}[v], \ \ u^- = -2\frac{\sqrt{h}-1}{\varepsilon} + \mathrm{F}_{\mu}^{-1}[v],
    \end{align}
    where $h = 1 +\varepsilon \zeta$, satisfy
    \begin{align}\label{Almost diagonalized WB system}
        \begin{cases}
            \partial_t u^+ + (\varepsilon\frac{3 u^+ + u^-}{4}+1) \mathrm{F}_{\mu}\partial_x u^+ = O(\mu\varepsilon), \\
            \partial_t u^- + (\varepsilon\frac{u^+ + 3u^-}{4}-1) \mathrm{F}_{\mu}\partial_x u^- = O(\mu\varepsilon).
        \end{cases}
    \end{align}
    where $\mathrm{F}_{\mu} = \sqrt{\frac{\tanh{(\sqrt{\mu}|D|)}}{\sqrt{\mu}|D|}}$. 
\end{mypp}

Before proving this result, we recall some definition and classical results we need.

\begin{mydef}\label{Definition Operateur pseudo}
    Let $f(x) \in \mathcal{C}^{n}(\mathbb{R})$ with $n \in \mathbb{N^*}$ be such that for any $0 \leq \beta \leq n$, $|\partial_x^{\beta} f(x) |_{L^{\infty}} \leq C_{\beta}$, where $C_{\beta} > 0$ is a constant depending on $\beta$. Let also $u(\xi) \in \mathcal{C}^{\infty}(\mathbb{R})$ be a smooth function such that there exists $m \in \mathbb{Z}$ for which $\forall \beta \geq 0$, $|\partial_{\xi}^{\beta} u(\xi)| \leq \widetilde{C}_{\beta} (1+ |\xi|)^{m-|\beta|}$, where $\widetilde{C}_{\beta} > 0$ is a constant depending on $\beta$.
    
    Then we define the operator $Op(f u)$ as follow: for any sufficiently regular functions $v$ we have
    \begin{align*}
        Op(f u)v(x) \colonequals f(x) \mathrm{u}(D)[v](x),
    \end{align*}
    where $\rm{u}(D)$ is the Fourier multiplier associated with the function $u$. 
\end{mydef}

\begin{myrem}
    The function $F_{\mu}(\xi) = \sqrt{\frac{\tanh{(\sqrt{\mu}|\xi|)}}{\sqrt{\mu}|\xi|}}$ defining the Fourier multiplier $\mathrm{F}_{\mu}$ satisfies the condition in the previous Definition \ref{Definition Operateur pseudo}. Idem for its inverse or its square.
\end{myrem}

\begin{mylem}\label{commutator estimates}
     Let $\mu_{\max} >0$. Let $(\varepsilon, \mu) \in \mathcal{A}=\{(\varepsilon,\mu), \ \ 0 \leq \varepsilon \leq 1, \ \ 0\leq \mu\leq \mu_{\max} \}$. Let $\mathrm{G}_{\mu}$ be a Fourier multiplier for which there exists $n \in \mathbb{N}^*$ such that for any $\beta \geq 0$ and $u \in H^{\beta+n}(\mathbb{R})$, $|(\mathrm{G}_{\mu}-1)[u]|_{H^{\beta}} \lesssim \mu |u|_{H^{\beta+n}}$. 
     \begin{itemize}
         \item Let $\alpha \geq 0$. Let also $f$ be a function in $H^{\alpha+n}(\mathbb{R})$. Then for any $u \in H^{\alpha+n}(\mathbb{R})$, the commutator $[\mathrm{G}_{\mu},f]u \colonequals \mathrm{G}_{\mu}[f u] - f \mathrm{G}_{\mu}[u]$ satisfies
     \begin{align*}
         |[\mathrm{G}_{\mu},f]u|_{H^{\alpha}} = |[\mathrm{G}_{\mu}-1,f]u|_{H^{\alpha}} \lesssim \mu |f|_{H^{\alpha+n}}|u|_{H^{\alpha+n}}.
     \end{align*}
     \item Let $\alpha \geq 0$. Let also $g$ be a function such that $\frac{g-1}{\varepsilon}$ is uniformly bounded in $\varepsilon$ in $H^{\alpha+n}(\mathbb{R})$. Then for any function $u \in H^{\alpha+n}(\mathbb{R})$,  
     \begin{align*}
         |[\mathrm{G}_{\mu},g]u|_{H^{\alpha}} = |[\mathrm{G}_{\mu}-1,g-1]u|_{H^{\alpha}} \lesssim \mu \varepsilon |\frac{g-1}{\varepsilon}|_{H^{\alpha+n}}|u|_{H^{\alpha+n}}.
     \end{align*}
     \end{itemize}
\end{mylem}
\begin{proof}
    Use product estimates \ref{product estimate}.
\end{proof}

\begin{myrem}\label{remark commutator estimates}
    The Fourier multiplier $\mathrm{F}_{\mu}$ satisfies the assumption of Lemma \ref{commutator estimates}. Idem for its inverse or its square.
    The functions $h$, $\sqrt{h}$ and its inverse satisfy the assumption on $g$.
\end{myrem}
\begin{proof}
The point on the Fourier multipliers $\mathrm{F}_{\mu}$, its inverse and its square is obvious. Idem for the function $h$. 
To prove the point on $\sqrt{h}$, we just need to use composition estimates \ref{Composition estimate} with $G(x) = \sqrt{1+x}-1$.

We now deal with $\frac{1}{\sqrt{h}}$. We remark that $\frac{1}{\sqrt{h}}-1 = \frac{1-\sqrt{h}}{1 + \sqrt{h}-1}$. Then we use quotient estimates \ref{Quotient estimate} ($\sqrt{h} \geq \sqrt{h_{\min}}$ because $\zeta$ satisfies the non-cavitation hypothesis \eqref{Non-Cavitation Hypothesis}). We get
\begin{align*}
    |\frac{1}{\sqrt{h}}-1|_{H^{\alpha}} \leq C(\frac{1}{\sqrt{h_{\min}}},|\sqrt{h}-1|_{H^{\max{(t_0,\alpha)}}})|\sqrt{h}-1|_{H^{\alpha}}.
\end{align*}
\end{proof}

We now prove Proposition \ref{Diag WB section 2}.
\begin{proof}
    We can write system \eqref{Whitham-Boussinesq system} under the form
\begin{align*}
    \partial_t U + A(U) \partial_x U = 0,
\end{align*}
where $U \colonequals \begin{pmatrix} \zeta \\ v \end{pmatrix}$ and $A(U) \colonequals \begin{pmatrix} \varepsilon v & h \\ \mathrm{F}_{\mu}^2[\circ] & \varepsilon v \end{pmatrix} $, with $h \colonequals 1 + \varepsilon \zeta$. 

The symbol of operator $A(U)$ is $A(U,\xi) = \begin{pmatrix} \varepsilon v & h \\ F_{\mu}^2(\xi) & \varepsilon v \end{pmatrix}$, $\xi \in \mathbb{R}$ being the frequency variable. So that for any smooth functions $W = (w_1, w_2)^T$, for all $x\in\mathbb{R}$ we have 
\begin{align*}
    A(U)W(t,x) = \Op(A(U,\xi))W(t,x) = \begin{pmatrix} \varepsilon v(t,x) w_1(t,x) + h(t,x) w_2(t,x) \\ \mathrm{F}_{\mu}^2 [w_1](t,x) + \varepsilon v(t,x) w_2(t,x) \end{pmatrix}.
\end{align*}

Now, we diagonalize $A(U,\xi)$.
\begin{align*}
    \det(A(U,\xi) - \lambda I_2) = (\varepsilon v - \lambda)^2 - F_{\mu}^2(\xi) h,
\end{align*}
so that 
\begin{align*}
    A(U,\xi) = P D P^{-1},
\end{align*}
with
\begin{align*}
    P = \begin{pmatrix} \frac{\sqrt{h}}{F_{\mu}(\xi)} & -\frac{\sqrt{h}}{F_{\mu}(\xi)} \\ 1 & 1 \end{pmatrix}, \ \ P^{-1} = \begin{pmatrix} \frac{F_{\mu}(\xi)}{2\sqrt{h}} & \frac{1}{2} \\ -\frac{F_{\mu}(\xi)}{2\sqrt{h}} & \frac{1}{2} \end{pmatrix} \ \ \et \ \ D = \begin{pmatrix} \varepsilon v + F_{\mu}(\xi)\sqrt{h} & 0 \\ 0 & \varepsilon v - F_{\mu}(\xi)\sqrt{h} \end{pmatrix}.
\end{align*}

Based on the symbolic calculus we apply $\mathrm{F}_{\mu}^{-1} \Op(P^{-1})$ to the system:
\begin{multline}\label{first step diagonalization}
    \mathrm{F}_{\mu}^{-1} \Op(P^{-1}) \partial_t U + \mathrm{F}_{\mu}^{-1} \Op(P^{-1}) \Op(A(U,\xi))\Op(P P^{-1}) \partial_x U = 0, \\
    \iff \mathrm{F}_{\mu}^{-1} \Op(P^{-1}) \partial_t U + \mathrm{F}_{\mu}^{-1} \Op(D) \mathrm{F}_{\mu} \mathrm{F}_{\mu}^{-1} \Op(P^{-1}) \partial_x U = R_1, 
\end{multline}
where 
\begin{align*}
    R_1 = &\mathrm{F}_{\mu}^{-1}(\Op(P^{-1}A(U,\xi))-\Op(P^{-1})\Op(A(U,\xi)))\partial_x U \\
    + &\mathrm{F}_{\mu}^{-1} \Op(P^{-1}A(U,\xi)) (\Op(P)\Op(P^{-1}) - \Op(P P^{-1}))\partial_x U \\
    + &\mathrm{F}_{\mu}^{-1} (\Op(D) - \Op(P^{-1}A(U,\xi))\Op(P)) \Op(P^{-1})\partial_x U.
\end{align*}
We remark that for each operators $\mathrm{F}_{\mu}^{-1}$, $\partial_x$, $\Op(P^{-1}A(U,\xi))$ and $\Op(P^{-1})$ there exists $n \in \mathbb{N}$ such that for any $\beta \geq 0$, the operator is bounded from $H^{\beta+n}(\mathbb{R})$ to $H^{\beta}(\mathbb{R})$. 
\begin{mylem} 
    $R_1$ is of order $O(\mu\varepsilon)$.
\end{mylem}
\begin{proof}
There are three terms in $R_1$. We deal with each of those separately. 
\begin{itemize}
    \item First term: for any $W = (w_1,w_2) \in H^{\beta}(\mathbb{R}) \times H^{\beta}(\mathbb{R})$ with $\beta \geq 0$,
    \begin{align*}
        (\Op(P^{-1}A(U,\xi)) - \Op(P^{-1})\Op(A(U,\xi)))W &= \frac{1}{2\sqrt{h}} \begin{pmatrix} -[\mathrm{F}_{\mu},\varepsilon v]w_1 - [\mathrm{F}_{\mu},h]w_2 \\ [\mathrm{F}_{\mu},\varepsilon v]w_1 + [\mathrm{F}_{\mu},h]w_2 \end{pmatrix}.
    \end{align*}
    Using Lemma \ref{commutator estimates} and $h \geq h_{\min}$, we have
    \begin{align*}
        \frac{1}{2\sqrt{h}}[\mathrm{F}_{\mu},\varepsilon v]w_1 = O(\mu\varepsilon) \ \ \et \ \ \frac{1}{2\sqrt{h}}[\mathrm{F}_{\mu},h]w_2 = O(\mu\varepsilon).
    \end{align*}

    \item Second term: for any $W = (w_1,w_2) \in H^{\beta}(\mathbb{R}) \times H^{\beta}(\mathbb{R})$ with $\beta \geq 0$,
    \begin{align*}
        (\Op(P)\Op(P^{-1})-I_d)W = \begin{pmatrix} \sqrt{h} [\mathrm{F}_{\mu}^{-1}, \frac{1}{\sqrt{h}}] \mathrm{F}_{\mu}[w_1] \\ w_2 \end{pmatrix}. 
    \end{align*}
    Using Lemma \ref{commutator estimates} and the boundedness of $\mathrm{F}_{\mu}$ and $\sqrt{h}$ we have
    \begin{align*}
        \sqrt{h} [\mathrm{F}_{\mu}^{-1}, \frac{1}{\sqrt{h}}] \mathrm{F}_{\mu}[w_1] = O(\mu\varepsilon).
    \end{align*}
    \item Third term: for any $W = (w_1,w_2) \in H^{\beta}(\mathbb{R}) \times H^{\beta}(\mathbb{R})$ with $\beta \geq 0$,
    \begin{multline*}
        (\Op(D)-\Op(P^{-1}A(U,\xi))\Op(P))W \\
        = \begin{pmatrix} -\frac{\varepsilon v}{2\sqrt{h}} [\mathrm{F}_{\mu}-1,\sqrt{h}]\mathrm{F}_{\mu}^{-1}[w_1-w_2] - \frac{1}{2}[\mathrm{F}_{\mu}^2,\sqrt{h}]\mathrm{F}_{\mu}^{-1}[w_1-w_2] \\ \frac{\varepsilon v}{2\sqrt{h}} [\mathrm{F}_{\mu}-1,\sqrt{h}]\mathrm{F}_{\mu}^{-1}[w_1-w_2] - \frac{1}{2}[\mathrm{F}_{\mu}^2,\sqrt{h}]\mathrm{F}_{\mu}^{-1}[w_1-w_2] \end{pmatrix} 
    \end{multline*}
    Using Lemma \ref{commutator estimates}, $h \geq h_{\min}$ and the boundedness of $\mathrm{F}_{\mu}^{-1}$, we have
    \begin{align*}
        \begin{cases}
            \frac{\varepsilon v}{2\sqrt{h}} [\mathrm{F}_{\mu}-1,\sqrt{h}]\mathrm{F}_{\mu}^{-1}[w_1-w_2] = O(\mu\varepsilon), \\
            \frac{1}{2}[\mathrm{F}_{\mu}^2,\sqrt{h}]\mathrm{F}_{\mu}^{-1}[w_1-w_2] = O(\mu\varepsilon).
        \end{cases}
    \end{align*}
\end{itemize}
\end{proof}
We continue the proof of Proposition \ref{Diag WB section 2}. We compute $\mathrm{F}_{\mu}^{-1}\Op(P^{-1})\partial_t U$:
\begin{align*}
    \mathrm{F}_{\mu}^{-1}\Op(P^{-1})\partial_t U = \mathrm{F}_{\mu}^{-1} \begin{pmatrix} \frac{1}{2\sqrt{h}} \mathrm{F}_{\mu}[\partial_t \zeta] + \frac{\partial_t v}{2} \\ -\frac{1}{2\sqrt{h}} \mathrm{F}_{\mu}[\partial_t \zeta] + \frac{\partial_t v}{2} \end{pmatrix} = \begin{pmatrix} \partial_t (\frac{1}{\varepsilon}\sqrt{h} + \frac{1}{2}\mathrm{F}_{\mu}^{-1}[v]) \\ \partial_t (-\frac{1}{\varepsilon} \sqrt{h} + \frac{1}{2}\mathrm{F}_{\mu}^{-1}[v]) \end{pmatrix} + R_2,
\end{align*}
where $R_2 = \mathrm{F}_{\mu}^{-1} \begin{pmatrix} - [\mathrm{F}_{\mu},\frac{1}{2\sqrt{h}}]\partial_t \zeta \\ [\mathrm{F}_{\mu},\frac{1}{2\sqrt{h}}]\partial_t \zeta \end{pmatrix}$. Here, to prove that $R_2$ is of order $O(\mu\varepsilon)$, in addition of Lemma \ref{commutator estimates} we need a control of $|\partial_t\zeta|_{H^{\beta}}$ for any $\beta \geq 0$. We get it using the first equation of \eqref{Whitham-Boussinesq system} and product estimates \ref{product estimate}. 

\indent

The same computations for $\mathrm{F}_{\mu}^{-1}\Op(P^{-1})\partial_x U$ gives
\begin{align*}
    \mathrm{F}_{\mu}^{-1}\Op(P^{-1})\partial_x U =
    \begin{pmatrix} \partial_x (\frac{1}{\varepsilon}\sqrt{h} + \frac{1}{2}\mathrm{F}_{\mu}^{-1}[v]) \\ \partial_x (-\frac{1}{\varepsilon}\sqrt{h} + \frac{1}{2}\mathrm{F}_{\mu}^{-1}[v]) \end{pmatrix} + R_3,
\end{align*}
where $R_3 = \mathrm{F}_{\mu}^{-1} \begin{pmatrix} - [\mathrm{F}_{\mu},\frac{1}{2\sqrt{h}}]\partial_x \zeta \\ [\mathrm{F}_{\mu},\frac{1}{2\sqrt{h}}]\partial_x \zeta \end{pmatrix}$ is of order $O(\mu\varepsilon)$.

So that 
\begin{multline}
    \mathrm{F}_{\mu}^{-1} \Op(D) \mathrm{F}_{\mu}\mathrm{F}_{\mu}^{-1}\Op(P^{-1})\partial_x U
    \\
    = \begin{pmatrix} \varepsilon \mathrm{F}_{\mu}^{-1}[v\mathrm{F}_{\mu}[\circ]] + \mathrm{F}_{\mu}^{-1}[\sqrt{h} \mathrm{F}_{\mu}^2[\circ]] & 0 \\ 0 & \varepsilon \mathrm{F}_{\mu}^{-1} [v \mathrm{F}_{\mu}[\circ]] - \mathrm{F}_{\mu}^{-1}[\sqrt{h}\mathrm{F}_{\mu}^2[\circ]] \end{pmatrix}\\ \times \begin{pmatrix} \partial_x (\frac{1}{\varepsilon}\sqrt{h} + \frac{1}{2}\mathrm{F}_{\mu}^{-1}[v]) \\ \partial_x (-\frac{1}{\varepsilon}\sqrt{h} + \frac{1}{2}\mathrm{F}_{\mu}^{-1}[v]) \end{pmatrix} + R_4,
\end{multline}
where $R_4 = \mathrm{F}_{\mu}^{-1} \Op(D) \mathrm{F}_{\mu} [R_3] = O(\mu\varepsilon)$ because of the boundedness of the operators $\mathrm{F}_{\mu}^{-1}, \Op(D)$ and $\mathrm{F}_{\mu}$.

Hence, we get
\begin{align*}
    \begin{cases}
        \partial_t(\frac{1}{\varepsilon}\sqrt{h} + \frac{1}{2} \mathrm{F}_{\mu}^{-1}[v]) + \big( \varepsilon \mathrm{F}_{\mu}^{-1}[v\circ] + \mathrm{F}_{\mu}^{-1}[\sqrt{h} \mathrm{F}_{\mu}[\circ]] \big) \mathrm{F}_{\mu}\partial_x(\frac{1}{\varepsilon}\sqrt{h} + \frac{1}{2} \mathrm{F}_{\mu}^{-1}[v]) = R_5, \\
        \partial_t (-\frac{1}{\varepsilon}\sqrt{h} + \frac{1}{2} \mathrm{F}_{\mu}^{-1}[v]) + \big( \varepsilon \mathrm{F}_{\mu}^{-1} [v \circ] - \mathrm{F}_{\mu}^{-1}[\sqrt{h}\mathrm{F}_{\mu}[\circ]] \big) \mathrm{F}_{\mu}\partial_x(-\frac{1}{\varepsilon}\sqrt{h} + \frac{1}{2} \mathrm{F}_{\mu}^{-1}[v]) = R_6,
    \end{cases}
\end{align*}
where $R_5$ and $R_6$ are combinations of $R_1$, $R_2$, $R_3$, $R_4$. They are of order $O(\mu\varepsilon)$. See also that $\frac{1}{\epsilon} \partial_x \sqrt{h}$ and $\frac{1}{\epsilon} \partial_x \sqrt{h}$ are of order $O(1)$

But by Lemma \ref{commutator estimates}, we know that for any $w \in H^{\beta}(\mathbb{R})$ with $\beta \geq 0$, we have
\begin{align*}
    [\mathrm{F}_{\mu},\sqrt{h}]w = O(\mu\varepsilon) \ \ \et \ \ [\mathrm{F}_{\mu}^{-1},v]w = O(\mu).
\end{align*}
So
\begin{align*}
    \begin{cases}
        \partial_t(\frac{\sqrt{h}-1}{\varepsilon} +  \frac{1}{2}\mathrm{F}_{\mu}^{-1}[v]) + \big( \varepsilon \mathrm{F}_{\mu}^{-1}[v] + \sqrt{h} \big) \mathrm{F}_{\mu}\partial_x(\frac{\sqrt{h}-1}{\varepsilon} + \frac{1}{2}\mathrm{F}_{\mu}^{-1}[v]) = R_7, \\
        \partial_t (-\frac{\sqrt{h}-1}{\varepsilon} + \frac{1}{2}\mathrm{F}_{\mu}^{-1}[v]) + \big( \varepsilon \mathrm{F}_{\mu}^{-1} [v] - \sqrt{h} \big) \mathrm{F}_{\mu}\partial_x(-\frac{\sqrt{h}-1}{\varepsilon} + \frac{1}{2}\mathrm{F}_{\mu}^{-1}[v]) = R_8,
    \end{cases}
\end{align*}
where $R_7$ and $R_8$ are of order $O(\mu\varepsilon)$.

We now denote 
\begin{align*}
    u^+ = \frac{\sqrt{h}-1}{\varepsilon} + \frac{1}{2}\mathrm{F}_{\mu}^{-1}[v], \ \ u^- = -\frac{\sqrt{h}-1}{\varepsilon} + \frac{1}{2}\mathrm{F}_{\mu}^{-1}[v].
\end{align*}
We get from \eqref{first step diagonalization} and the above:
\begin{align*}
    \begin{cases}
        \partial_t u^+ + (\varepsilon\frac{3 u^+ + u^-}{2}+1) \mathrm{F}_{\mu}\partial_x u^+ = O(\mu\varepsilon), \\
        \partial_t u^- + (\varepsilon\frac{u^+ + 3u^-}{2}-1) \mathrm{F}_{\mu}\partial_x u^- = O(\mu\varepsilon).
    \end{cases}
\end{align*}
\end{proof}

\label{Formal diagonalization of the Whitham-Boussinesq equation}

\subsection{From the diagonalized Whitham-Boussinesq equations to the Whitham equations}

In this subsection we prove Proposition \ref{First main result}.

To prove the first part of the proposition, i.e. there exists a time $T > 0$ such that for all times $t \in [0, \frac{T}{\epsilon}]$, $u^-(t,\cdot)$ and $u^+(t,\cdot)$ are of order $O(\mu)$, we need local well-posedness and stability results on the exact diagonalized system
\begin{align}\label{Diagonalized WB equations}
        \begin{cases}
            \partial_t u^+_e + (\varepsilon\frac{3 u^+_e + u^-_e}{2}+1) \mathrm{F}_{\mu}\partial_x u^+_e = 0, \\
            \partial_t u^-_e + (\varepsilon\frac{u^+_e + 3u^-_e}{2}-1) \mathrm{F}_{\mu}\partial_x u^-_e = 0.
        \end{cases}
    \end{align}

We begin by proving the local well-posedness of system \eqref{Diagonalized WB equations}.
\begin{mypp}\label{local existence}(local existence)
    Let $(\mu,\epsilon) \in \mathcal{A}$ (see Definition \ref{shallow water regime}). Let $1/2 < t_0 \leq 1$, $\alpha \geq t_0+1$ and $u^+_{0,e}, u^-_{0,e} \in H^{\alpha}(\mathbb{R})$. Then there exists a time $T > 0$ such that system \eqref{Diagonalized WB equations}
    admits a unique solution $(u^+_e, u^-_e) \in \mathcal{C}([0,\frac{T}{\varepsilon}];H^{\alpha}(\mathbb{R})^2)$ with initial conditions $(u^+_{0,e},u^-_{0,e})$.

Moreover $\frac{1}{T}$ and $\sup_{0 \leq t \leq \frac{T}{\varepsilon}} |u^{+}_e|_{H^{\alpha}} + |u^-_e|_{H^{\alpha}}$ can be estimated by $|u^+_{0,e}|_{H^{\alpha}}$, $|u^-_{0,e}|_{H^{\alpha}}$,  $t_0$, $\mu_{\max}$ (and are independent of the parameters $(\mu,\varepsilon) \in \mathcal{A}$).
\end{mypp}
\begin{proof}
    System \eqref{Diagonalized WB equations} is very similar to symmetric quasilinear hyperbolic systems. The only difference is the presence of the non-local operator $\mathrm{F}_{\mu}$. The well-posedness of such systems relies on the energy estimates (a priori estimates). We will establish them here. For the rest of the proof, one can follow the one in \cite{Metivier2008} where the author uses a method of regularisation suitable for our system .   
    
    First, remark that if we define $U_e(t) \colonequals \begin{pmatrix} u^+_e(t) \\ u^-_e(t) \end{pmatrix}$ and 
    \begin{align}\label{expression of A}
        A(U_e) \colonequals \begin{pmatrix} (\varepsilon \frac{3u^+_e + u^-_e}{2} +1) & 0 \\ 0 & (\varepsilon \frac{u^+_e + 3u^-_e}{2}-1) \end{pmatrix}.
    \end{align}
    Then system \eqref{Diagonalized WB equations} can be written under the form
    \begin{align}\label{matrix form of diagonalized system}
        \partial_t U_e + A(U_e)\mathrm{F}_{\mu}\partial_x U_e = 0.
    \end{align}
    It justifies the following framework for energy estimates.
    Let $\beta \geq 0$, $R \in L^{\infty}([0,\frac{T}{\varepsilon}]; H^{\beta}(\mathbb{R})^2)$ and $\underline{U} \in L^{\infty}([0,\frac{T}{\varepsilon}]; H^{\max{(\beta,t_0+1)}}(\mathbb{R})^2)$. Let $U \in \mathcal{C}^0([0,\frac{T}{\varepsilon}];H^{\beta}(\mathbb{R})^2)\cap \mathcal{C}^1([O,\frac{T}{\varepsilon}];H^{\beta+1}(\mathbb{R})^2)$ solves the equation 
    \begin{align}\label{generic equation for energy estimates}
        \partial_t U + A(\underline{U})\mathrm{F}_{\mu}\partial_x U = \varepsilon R.
    \end{align}
    
    \begin{mylem}\label{Energy estimates of order 0}(Energy estimates of order 0)
        For all times $t \in [0, \frac{T}{\varepsilon}]$,
        \begin{align*}
            |U(t)|_2 \leq e^{\varepsilon \delta_0 t}|U(0)|_2 + \varepsilon \gamma_0 \int_{0}^{t} |R(t')|_2 \dt',
        \end{align*}
        where $\delta_0, \gamma_0  = C(T,|\underline{U}|_{L^{\infty}([0,\frac{T}{\varepsilon}]; H^{t_0 +1})})$
    \end{mylem}
    \begin{proof}
        In what follows we will write $(\cdot,\cdot)_2$ the scalar product in $L^2(\mathbb{R})^2$.
        Using the fact that $U$ solves equation \eqref{generic equation for energy estimates} we get
        \begin{align*}
            \frac{1}{2} \frac{\mathrm{d}}{\dt} |U|_2^2 = (\partial_t U,U)_2 &= -(A(\underline{U})\mathrm{F}_{\mu}\partial_x U, U)_2 + \varepsilon (R,U)_2.
        \end{align*}
        But $A(\underline{U})$ is symmetric, so
        \begin{align*}
            (A(\underline{U})\mathrm{F}_{\mu}\partial_x U, U)_2 &= (\partial_x \mathrm{F}_{\mu}U, A(\underline{U})U)_2 = -(\mathrm{F}_{\mu}U,(\partial_x A(\underline{U}))U)_2 - (U,\mathrm{F}_{\mu} [A(\underline{U})\partial_x U])_2 \\
            &= (\mathrm{F}_{\mu}U,(\partial_x A(\underline{U}))U)_2 - (A(\underline{U})\mathrm{F}_{\mu}\partial_x U, U)_2 - ([\mathrm{F}_{\mu},A(\underline{U})]\partial_x U, U)_2,
        \end{align*}
        where $\partial_x A(\underline{U}) = \begin{pmatrix} \varepsilon \frac{3 \partial_x \underline{u_1} +\partial_x \underline{u_2}}{2} & 0 \\ 0 & \varepsilon \frac{ \partial_x \underline{u_1} +3\partial_x \underline{u_2}}{2} \end{pmatrix}$ and
        \begin{align*}
            [\mathrm{F}_{\mu},A(\underline{U})] = \begin{pmatrix} [\mathrm{F}_{\mu},\varepsilon \frac{3 \underline{u}_1 + \underline{u}_2}{2}+1] & 0 \\ 0 & [\mathrm{F}_{\mu},\varepsilon \frac{ \underline{u}_1 +3 \underline{u}_2}{2}-1] \end{pmatrix} = \varepsilon \begin{pmatrix} [\mathrm{F}_{\mu}, \frac{3 \underline{u}_1 + \underline{u}_2}{2}] & 0 \\ 0 & [\mathrm{F}_{\mu}, \frac{ \underline{u}_1 +3 \underline{u}_2}{2}] \end{pmatrix}.
        \end{align*}
    
        Hence
        \begin{align}\label{equation norm U energy estimates order 0}
            \frac{1}{2} \frac{\mathrm{d}}{\dt} |U|_2^2 = \frac{1}{2}(\mathrm{F}_{\mu}U, (\partial_x A(\underline{U}))U)_2 + \frac{1}{2} ([\mathrm{F}_{\mu},A(\underline{U})]\partial_x U, U)_2+ \varepsilon (R,U)_2.
        \end{align}
        Then using Cauchy-Schwarz inequality we get
        \begin{align*}
            &|U|_2 \frac{\mathrm{d}}{\dt} |U|_2 \leq \frac{1}{2} |\partial_x(A(\underline{U}))U|_2|U|_2 + \frac{1}{2} |[\mathrm{F}_{\mu},A(\underline{U})]\partial_x U|_2 |U|_2 + \varepsilon |R|_2 |U|_2 \\
            \iff & \frac{\mathrm{d}}{\dt} |U|_2 \leq \frac{1}{2} |\partial_x(A(\underline{U}))U|_2 + \frac{1}{2} |[\mathrm{F}_{\mu},A(\underline{U})]\partial_x U|_2 + \varepsilon |R|_2.
        \end{align*}
        But by product estimates \ref{product estimate} we have
        \begin{align}\label{estimates on partial x A}
            |\partial_x (A(\underline{U}))U|_2 \leq \varepsilon C(|\underline{U}|_{L^{\infty}_t H^{t_0+1}_x}) |U|_2.
        \end{align}
        Moreover by commutator estimates \ref{Commutator estimates 2.0} we also have
        \begin{align}\label{estimates on commutator F and A}
            |[\mathrm{F}_{\mu},A(\underline{U})]\partial_x U|_2 \leq \varepsilon C(|\underline{U}|_{L^{\infty}_t H^{t_0+1}_x}) |U|_2.
        \end{align}
        So that
        \begin{align*}
            \frac{\mathrm{d}}{\dt} |U|_2 &\leq \varepsilon C(|\underline{U}|_{L^{\infty}_t H^{t_0+1}_x}) |U|_2 + \varepsilon |R|_2 \leq \varepsilon \delta_0 |U|_2 + \varepsilon|R|_2,
        \end{align*}
        where $\delta_0 = C(|\underline{U}|_{L^{\infty}_t H^{t_0+1}_x})$.
        
        We integrate this inequality in time
        \begin{align*}
            |U|_2 &\leq e^{\varepsilon \delta_0 t}|U(0)|_2 + \varepsilon \int_{0}^{t} e^{ \varepsilon \delta_0 (t-t')} |R|_2 \dt' \\
            &\leq e^{\varepsilon \delta_0 t}|U(0)|_2 + \varepsilon e^{\delta_0 T} \int_{0}^{t}  |R|_2 \dt'.
        \end{align*}
        It is the thesis with $\gamma_0 = e^{\delta_0 T}$.
       
    \end{proof}
    Using this lemma, we get the energy estimates of order $\beta$.
    \begin{mylem}\label{Energy estimates of order beta}(Energy estimates of order $\beta$)
        For all times $t \in [0,\frac{T}{\varepsilon}]$,
         \begin{align*}
            |U(t)|_{H^{\beta}} \leq e^{\varepsilon \delta_{\beta} t}|U(0)|_{H^{\beta}} + \varepsilon \gamma_{\beta} \int_{0}^{t} |R(t')|_{H^{\beta}} \dt',
        \end{align*}
        where $\delta_{\beta}, \gamma_{\beta} = C(T,|\underline{U}|_{L^{\infty}([0,\frac{T}{\varepsilon}]; H^{\max{(\beta,t_0+1)}})})$.
    \end{mylem}
    \begin{proof}
        Let $\Lambda^{\beta} \colonequals (1-\partial_x^2)^{\beta/2}$. Then $|\Lambda^{\beta}U|_2$ is equivalent to $|U|_{H^{\beta}}$. Applying $\Lambda^\beta$ to \eqref{generic equation for energy estimates} we have  
        \begin{align*}
            &\Lambda^{\beta}\partial_t U + \Lambda^{\beta}[A(\underline{U})\mathrm{F}_{\mu}\partial_x U] = \varepsilon \Lambda^{\beta}R \\
            \iff &\partial_t \Lambda^{\beta} U + A(\underline{U})\mathrm{F}_{\mu}\partial_x \Lambda^{\beta} U = \varepsilon \Lambda^{\beta}R + \varepsilon B_{\beta}(\underline{U})\mathrm{F}_{\mu}\partial_x U,
        \end{align*}
        where 
        \begin{align*}
            B_{\beta}(\underline{U}) \colonequals \frac{1}{\varepsilon}\begin{pmatrix} [\Lambda^{\beta},\varepsilon \frac{3 \underline{u}_1 + \underline{u}_2}{2}+1] & 0 \\ 0 & [\Lambda^{\beta},\varepsilon \frac{ \underline{u}_1 +3 \underline{u}_2}{2}-1] \end{pmatrix} = \begin{pmatrix} [\Lambda^{\beta},\frac{3 \underline{u}_1 + \underline{u}_2}{2}] & 0 \\ 0 & [\Lambda^{\beta}, \frac{ \underline{u}_1 +3 \underline{u}_2}{2}] \end{pmatrix}.
        \end{align*}
        Using commutator estimates \ref{Commutator estimates 2.0} and the boundedness of $\mathrm{F}_{\mu}$ we get
        \begin{align*}
            |B_{\beta}(\underline{U})\mathrm{F}_{\mu}\partial_x U|_2 \lesssim |\underline{U}|_{L^{\infty}_t H^{\max{(t_0+1, \beta)}}_x}|U|_{H^{\beta}}.
        \end{align*}
        Then we use the energy estimates of order 0 (see Lemma \ref{Energy estimates of order 0}) to get 
        \begin{align*}
            |U|_{H^{\beta}} &\leq e^{\varepsilon \delta_0 t}|U(0)|_{H^{\beta}} + \varepsilon \gamma_0 \int_{0}^{t} |R|_{H^{\beta}} + \varepsilon \gamma_0 \int_0^t |B_{\beta}(\underline{U})\partial_x U|_2 \dt' \\
            &\lesssim e^{\varepsilon \delta_0 t}|U(0)|_{H^{\beta}} + \varepsilon C(T,|\underline{U}|_{L^{\infty}_t H^{\max{(t_0+1, \beta)}}_x}) \int_{0}^{t} |U|_{H^{\beta}} \dt'+ \varepsilon \gamma_0 \int_{0}^{t} |R|_{H^{\beta}} \dt'.
        \end{align*}
        It only remains to use Grönwall's lemma to get the result.
    \end{proof}
    This concludes the proof of Proposition \ref{local existence}.
\end{proof}

\begin{mycoro}\label{persistence of smallness}
    Let $t_0 > 1/2$ and $\alpha \geq t_0 +1$. Let $U_e = \begin{pmatrix} u_e^+ \\ u_e^- \end{pmatrix} \in \mathcal{C}([0,\frac{T}{\varepsilon}],H^{\alpha}(\mathbb{R})^2)$ solves the exact diagonalized system \eqref{Diagonalized WB equations} with $u^-_e(0) = O(\mu)$, then for all times $[0,\frac{T}{\varepsilon}]$, $u^-_e = O(\mu)$. If instead we take $u^+_e(0) = O(\mu)$ then for all times $[0,\frac{T}{\varepsilon}]$, $u^+_e = O(\mu)$.
\end{mycoro}
\begin{proof}
    Using energy estimates of Lemma \ref{Energy estimates of order beta} on the equation on $u^-_e$ of system \eqref{Diagonalized WB equations}, we get
    \begin{align*}
        |u_e^-|_{H^{\alpha}} \leq C(T,|U_e|_{L^{\infty}_t H^{\alpha}_x}) |u_e^-(0)|_{H^{\alpha}}. 
    \end{align*}
    It gives the result on $u_e^-$. We do the same for $u_e^+$ with the other equation of \eqref{Diagonalized WB equations}.
\end{proof}

\begin{mypp}(Stability)\label{Stability}
Let $t_0 > 1/2$ and $\alpha \geq t_0 +1$. Let $U_e = \begin{pmatrix} u_e^+ \\ u_e^- \end{pmatrix} \in \mathcal{C}([0,\frac{T}{\varepsilon}],H^{\alpha}(\mathbb{R})^2)$ solve the exact diagonalized system \eqref{Diagonalized WB equations}
and $U = \begin{pmatrix} u^+ \\ u^- \end{pmatrix} \in \mathcal{C}([0,\frac{T}{\varepsilon}],H^{\alpha}(\mathbb{R})^2)$, defined from sufficiently regular solutions of the water waves equations satisfying the non-cavitation hypothesis \eqref{Non-Cavitation Hypothesis} through the formulas \eqref{Riemann invariants}. By Proposition \ref{Diag WB section 2}, we have
\begin{align}\label{two equations}
    \begin{cases}
        \partial_t U_e + A(U_e) \mathrm{F}_{\mu}\partial_x U_e = 0, \\
        \partial_t U + A(U) \mathrm{F}_{\mu}\partial_x U = \mu\varepsilon R, 
    \end{cases}
\end{align}
where $A$ is defined by \eqref{expression of A} and $R \in L^{\infty}([0,\frac{T}{\varepsilon}],H^{\alpha}(\mathbb{R})^2)$ is uniformly bounded in $(\mu,\varepsilon) \in \mathcal{A}$.

Then for all times $t \in [0,\frac{T}{\varepsilon}]$, we have
\begin{align}\label{Stability inequality}
    |U_e -U|_{H^{\alpha}} \leq c_{\alpha}(|U_e(0) - U(0)|_{H^{\alpha}} + \mu\varepsilon t |R|_{L^{\infty}_t H^{\alpha}_x}),
\end{align}
where $c_{\alpha} = C(T,|U_e|_{L^{\infty}_t H^{\alpha}_x},|U|_{L^{\infty}_t H^{\alpha+1}_x})$.

In particular, if we take $U_e(0) = U(0)$ and $u^-(0) = O(\mu)$, then for all times $t \in [0,\frac{T}{\epsilon}]$,
\begin{align*}
    u^- = O(\mu).
\end{align*}
If instead, we take $u^+(0) = O(\mu)$, then for all times $t \in [0,\frac{T}{\epsilon}]$,
\begin{align*}
    u^+ = O(\mu).
\end{align*} 
\end{mypp}

\begin{proof}
    First we subtract the two equations of \eqref{two equations}, we get
    \begin{align*}
         \partial_t (U_e - U) + A(U_e) \mathrm{F}_{\mu}\partial_x (U_e - U) 
         = -\varepsilon (\mu R + \frac{1}{\epsilon}(A(U_e) - A(U))\mathrm{F}_{\mu}\partial_x U).
    \end{align*}
    We use the estimates of order $\alpha$ from Lemma \ref{Energy estimates of order beta}
    \begin{align*}
        |U_e - U|_{H^{\alpha}} \leq &e^{\varepsilon \delta_{\alpha} t}|U_e(0) - U(0)|_{H^{\alpha}} + \varepsilon \gamma_{\alpha} \int_{0}^{t} |\mu R(t')|_{H^{\alpha}} \dt' \\
        &+ \varepsilon \gamma_{\alpha} \int_0^t \frac{1}{\epsilon}|(A(U_e)-A(U))\mathrm{F}_{\mu}\partial_x U|_{H^{\alpha}} \dt'. 
    \end{align*}
    But using product estimates \ref{product estimate}, the algebra properties of $H^{\alpha}(\mathbb{R})$ and the boundedness of $\mathrm{F}_{\mu}$, we have
    \begin{align*}
        |(A(U_e)-A(U))\mathrm{F}_{\mu}\partial_x U|_{H^{\alpha}} \lesssim \epsilon |U_e - U|_{H^{\alpha}} |U|_{H^{\alpha +1}}.
    \end{align*}
    So that using Grönwall's lemma we get 
    \begin{align*}
        |U_e - U|_{H^{\alpha}} &\leq \Big[ e^{\varepsilon \delta_{\alpha} t}|U_e(0)- U(0)|_{H^{\alpha}} + \varepsilon \gamma_{\alpha} \int_{0}^{t} |\mu R(t')|_{H^{\alpha}} \dt' \Big] e^{\varepsilon \gamma_{\alpha} \int_0^t |U|_{H^{\alpha +1}} \dt'} \\
        &\leq C(T,|U_e|_{L^{\infty}_t H^{\alpha}_x})(|U_e(0) - U(0)|_{H^{\alpha}} + \mu\varepsilon t |R|_{L^{\infty}_t H^{\alpha}_x})e^{\varepsilon \gamma_{\alpha}t |U|_{L^{\infty}_t H^{\alpha+1}_x}},
    \end{align*}
    which is the first part of the result.
    
    Now if we suppose $U_e(0) = U(0)$ and $u^-(0) = O(\mu)$, then for all times $t \in [0,\frac{T}{\epsilon}]$, we have
    \begin{align*}
        |u^-|_{H^{\alpha}} \leq |u_e^- - u^-|_{H^{\alpha}} + |u_e^-|_{H^{\alpha}} \leq \mu c_{\alpha} T |R|_{L^{\infty}_t H^{\alpha}_x} + |u_e^-|_{H^{\alpha}}. 
    \end{align*}
    Using Corollary \ref{persistence of smallness} we get the result. We do the same if $u^+(0) = O(\mu)$.
\end{proof}

We now have all the elements to prove Proposition \ref{First main result}. 
\begin{proof}
    Supposing $u^-(0) = O(\mu)$, we have for all times $t \in [0,\frac{T}{\varepsilon}]$, $\varepsilon u^-(t) = O(\mu\varepsilon)$ (see Proposition \ref{Stability}). But $u^-$ and $u^+$ solve the first equation of \eqref{Almost diagonalized WB system}. So 
    \begin{align}\label{proof first main result}
        \partial_t u^+ + \mathrm{F}_{\mu}\partial_x u^+ + \frac{3\varepsilon}{2}u^+\partial_x u^+ = O(\mu\varepsilon). 
    \end{align}
    The last equivalence coming from product estimates \ref{product estimate} and $|(\mathrm{F}_{\mu}-1)[u]|_{H^{\beta}} \lesssim \mu |u|_{H^{\beta}}$ for any $\beta \geq 0$ and $u$ in $H^{\beta+2}(\mathbb{R})$.
    
    Supposing instead $u^+(0) = O(\mu)$, we have for all times $t \in [0,\frac{T}{\varepsilon}]$, $\varepsilon u^+(t) = O(\mu\varepsilon)$. Then, the second equation of \eqref{Almost diagonalized WB system} gives
    \begin{align*}
        \partial_t u^- - \mathrm{F}_{\mu}\partial_x u^- + \frac{3\varepsilon}{2}u^-\partial_x u^- = O(\mu\varepsilon).
    \end{align*}
\end{proof}

It only remains to prove Proposition \ref{Corollary convergence introduction}.
\begin{proof}
    Let $U_e = \begin{pmatrix} u_e^+ \\ u_e^- \end{pmatrix} \in \mathcal{C}([0,\frac{T}{\varepsilon}],H^{\alpha}(\mathbb{R})^2)$ solve the system \eqref{Diagonalized WB equations}, with the initial data $u^-_{e,0} = O(\mu)$ and $u^+_{e,0}$, defined by $(\zeta_0,\psi_0)$ through the formulas \eqref{Diagonalized system introduction}. By the energy estimates of Lemma \ref{Energy estimates of order beta} applied on the second equation of system \eqref{Diagonalized WB equations}, we get for all times $t \in [0,\frac{T}{\varepsilon}]$,
    \begin{align*}
        |u^-_e|_{H^{\alpha}} \leq C|u_e^-(0)|_{H^{\alpha}},
    \end{align*}
    where $C,T^{-1} = C(\frac{1}{h_{\min}}, \mu_{\max}, |\zeta_0|_{H^{\alpha+n}}, |\psi_0|_{\dot{H}^{\alpha+n+1}})$.
    So that $u^+ \in \mathcal{C}([0,\frac{T}{\varepsilon}],H^{\alpha}(\mathbb{R}))$, solution of the Whitham equation (which is well-posed in the Sobolev space $H^{\alpha}(\mathbb{R})$ \cite{EhrestromEscherPei2015})
    \begin{align*}
        \partial_t u^+ + \mathrm{F}_{\mu}\partial_x u^+ + \frac{3\varepsilon}{2} u^+ \partial_x u^+ = 0,
    \end{align*}
    satisfies (see \eqref{proof first main result})
    \begin{align*}
        \partial_t u^+ + (\varepsilon\frac{3 u^+ + u^-_e}{2}+1) \mathrm{F}_{\mu}\partial_x u^+ = \mu\varepsilon R,
    \end{align*}
    where $R$ is uniformly bounded in $(\mu,\varepsilon)$ in $H^{\alpha}(\mathbb{R})$. Using the stability estimates \eqref{Stability inequality}, with $U = \begin{pmatrix} u^+ \\ u^-_e \end{pmatrix}$ and $U_e(0) = U(0)$, we have for all times $t \in [0,\frac{T}{\varepsilon}]$,
    \begin{align*}
        |u^+_e - u^+|_{H^{\alpha}} \leq C \mu\varepsilon t,
    \end{align*}
    where $C,T^{-1} = C(\frac{1}{h_{\min}}, \mu_{\max}, |\zeta_0|_{H^{\alpha+n}}, |\psi_0|_{\dot{H}^{\alpha+n+1}})$.
    
    Thus by Corollary \ref{First corollary introduction} we have for all times $t \in [0,\frac{T}{\varepsilon}]$,
    \begin{align*}
        |(\zeta - \zeta_{\Wh,+},\psi - \psi_{\Wh,+})|_{H^{\alpha}\times \dot{H}^{\alpha+1}} \leq &|(\zeta - \zeta_{\rm{c}},\psi - \psi_{\rm{c}})|_{H^{\alpha}\times \dot{H}^{\alpha+1}} \\
        + &|(\zeta_{\rm{c}} - \zeta_{\Wh,+},\psi_{\rm{c}} - \psi_{\Wh,+})|_{H^{\alpha}\times \dot{H}^{\alpha+1}} \\
        \leq &C(|u_e^-|_{H^{\alpha+1}} + \mu\varepsilon t) \\
        \leq &C(|u_{e,0}^-|_{H^{\alpha+1}} + \mu \varepsilon t),
    \end{align*}
    where $C,T^{-1} = C(\frac{1}{h_{\min}}, \mu_{\max}, |\zeta_0|_{H^{\alpha+n}}, |\psi_0|_{\dot{H}^{\alpha+n+1}})$.
    
    The reasoning is the same if instead $u_e^+(0) = O(\mu)$.
\end{proof}

\label{From the diagonalized Whitham-Boussinesq system to the Whitham equations}

\section{Decoupling the water waves equations into two Whitham equations}
\label{Decoupling the water waves equations into two Whitham equations}
The goal of this section is to prove Theorem \ref{Final theorem introduction}.
In the previous section, we proved the consistency (see Proposition \ref{Consistency of the WB equations} for the definition) of the water waves equations with the Whitham equations at a precision order $O(\mu\varepsilon)$ in the shallow water regime $\mathcal{A} \colonequals \{ 0\leq \mu \leq \mu_{\max} , 0\leq \varepsilon \leq 1 \}$ (see Proposition \ref{Diag WB section 2}). For that purpose we made a restrictive hypothesis on the initial conditions. We supposed either $u^-(0) = O(\mu)$ or $u^+(0) = O(\mu)$ (see Notation \ref{Big O Sobolev spaces}). In this section we use a generalisation of Birkhoff's normal form algorithm for almost smooth Hamiltonians, introduced by Bambusi in \cite{Bambusi20}, on the water waves Hamiltonian. It will allow us to decouple the water waves equations into two Whitham equations satisfied by the two front waves, at a precision order $O(\mu\varepsilon + \varepsilon^2)$ in the shallow water regime, without any assumption of smallness on the initial conditions. 

We will only recall the elements of Bambusi's theory which are useful to understand our reasoning. For more details see \cite{Bambusi20}. 

In this section we always suppose the parameters $(\mu,\varepsilon)$ in $\mathcal{A}$. Moreover every $\alpha$ will be a positive integer.  

\subsection{From the water waves equations to a Whitham-Boussinesq system}
The starting point is the Hamiltonian of the water waves system
\begin{align}\label{HamiltonianWW}
    H_{\WW} = \frac{1}{2} \int_{\mathbb{R}} \zeta^2 \dx + \frac{1}{2} \int_{\mathbb{R}} \psi \frac{1}{\mu} \mathcal{G}^{\mu}[\varepsilon\zeta]\psi \dx.
\end{align}
We use the first shallow water estimates of Proposition 1.8 in \cite{Emerald2020}, which we recall here for the sake of clarity.
\begin{mypp}\label{shallow water approximation of the Dirichlet-Neumann operator}
    Let $\alpha \geq 0$, and $\zeta \in H^{\alpha+4}(\mathbb{R})$ be such that \eqref{Non-Cavitation Hypothesis} is satisfied. Let $\psi \in \dot{H}^{\alpha+3}(\mathbb{R})$, then
    \begin{align*}
        |\frac{1}{\mu}\mathcal{G}^{\mu}[\varepsilon\zeta]\psi + \partial_x(h \mathrm{F}_{\mu}^2[\nabla\psi])|_{H^{\alpha}} \leq \mu \varepsilon M(\alpha+4)|\psi|_{\dot{H}^{\alpha+4}},
    \end{align*}
    where $h = 1+\varepsilon\zeta$ and for any $\beta \geq 2 $, $M(\beta) \colonequals C(\frac{1}{h_{\min}},\mu_{\max},|\zeta|_{H^{\beta}})$.
\end{mypp}

\begin{mypp}\label{From HBX to HWW}
    Let $H_{\WW}$ be the Hamiltonian of the water waves equations \eqref{HamiltonianWW}. Then for any $\alpha \geq 0$ and any $(\zeta, \psi) \in H^{\alpha+4}(\mathbb{R}) \times \dot{H}^{\alpha+4}(\mathbb{R})$ we have
    \begin{align}\label{comparison between hamiltonians}
        |J\nabla H_{\WW} - J\nabla (H_0 + \varepsilon H_1) |_{H^{\alpha}\times \dot{H}^{\alpha+1}} < \mu\varepsilon M(\alpha+4)|\psi|_{\dot{H}^{\alpha+4}},
    \end{align}
    where $J = \begin{pmatrix} 0 & 1 \\ -1 & 0 \end{pmatrix}$, $H_0 \colonequals \frac{1}{2} \int_{\mathbb{R}} \zeta^2 \dx + \frac{1}{2} \int_{\mathbb{R}} (\mathrm{F}_{\mu}[\partial_x\psi])^2 \dx $ and $ H_1 \colonequals \frac{1}{2} \int_{\mathbb{R}} \zeta (\mathrm{F}_{\mu}[\partial_x \psi])^2 \dx$. 
\end{mypp}

\begin{proof}
   We defined $H_{\WW}$ as the Hamiltonian of the water waves equations, so
   \begin{align}\label{Truc1}
       J \nabla H_{\WW} = \begin{pmatrix} \frac{1}{\mu}\mathcal{G}^{\mu}[\varepsilon\zeta]\psi \\ - \zeta - \frac{\varepsilon}{2}(\partial_x\psi)^2 + \frac{\mu\varepsilon}{2}\frac{(\frac{1}{\mu}\mathcal{G}^{\mu}[\varepsilon\zeta]\psi + \varepsilon\partial_x\zeta\partial_x\psi)^2}{1+\varepsilon^2\mu(\partial_x\zeta)^2} \end{pmatrix}.
   \end{align}
   We can easily compute $J \nabla (H_0 + \varepsilon H_1)$, we get
   \begin{align}\label{Truc2}
       J \nabla (H_0 + \varepsilon H_1) = \begin{pmatrix} -\mathrm{F}_{\mu}^2[\partial_x^2 \psi] - \mathrm{F}_{\mu}[\partial_x(\varepsilon \zeta\mathrm{F}_{\mu}[\partial_x\psi])] \\ -\zeta - \frac{\varepsilon}{2} (\mathrm{F}_{\mu}[\partial_x\psi])^2 \end{pmatrix}.
   \end{align}
   We compare the first term of \eqref{Truc1} and \eqref{Truc2}.
   \begin{align*}
       &|\frac{1}{\mu}\mathcal{G}^{\mu}[\varepsilon\zeta]\psi + \mathrm{F}_{\mu}^2[\partial_x^2 \psi] + \mathrm{F}_{\mu}[\partial_x(\varepsilon \zeta\mathrm{F}_{\mu}[\partial_x\psi])]|_{H^{\alpha}} \\
       \leq &|\frac{1}{\mu}\mathcal{G}^{\mu}[\varepsilon\zeta]\psi + \partial_x(h \mathrm{F}_{\mu}^2[\partial_x \psi])|_{H^{\alpha}} + |(\mathrm{F}_{\mu}-1)[\partial_x(\varepsilon\zeta\mathrm{F}_{\mu}[\partial_x\psi])]|_{H^{\alpha}} + |\partial_x(\varepsilon\zeta (\mathrm{F}_{\mu}-1)\mathrm{F}_{\mu}[\partial_x\psi])|_{H^{\alpha}}.
   \end{align*}
   Hence, using Proposition \ref{shallow water approximation of the Dirichlet-Neumann operator}, the fact that for any $u \in H^{\alpha+2}(\mathbb{R})$, $|(\mathrm{F}_{\mu}-1)[u]|_{H^{\alpha}} \lesssim \mu|u|_{H^{\alpha+2}}$ and product estimates \ref{product estimate}, we get 
   \begin{align*}
       |\frac{1}{\mu}\mathcal{G}^{\mu}[\varepsilon\zeta]\psi + \mathrm{F}_{\mu}^2[\partial_x^2 \psi] + \mathrm{F}_{\mu}[\partial_x(\varepsilon \zeta\mathrm{F}_{\mu}[\partial_x\psi])]|_{H^{\alpha}} \lesssim \mu\varepsilon M(\alpha+4)|\psi|_{\dot{H}^{\alpha+4}}.
   \end{align*}
   Then we compare the second term of \eqref{Truc1} and \eqref{Truc2}. We first remark that there is a part of \eqref{Truc1} which is clearly of order $O(\mu\varepsilon)$, we start by estimate it using quotient estimates \ref{Quotient estimate}, product estimates \ref{product estimate} and Proposition \ref{Theorem 3.15}
   \begin{align*}
       |\frac{(\frac{1}{\mu}\mathcal{G}^{\mu}[\varepsilon\zeta]\psi + \varepsilon\partial_x\zeta\partial_x\psi)^2}{1+\varepsilon^2\mu(\partial_x\zeta)^2}|_{\dot{H}^{\alpha+1}} \leq &|\frac{(\frac{1}{\mu}\mathcal{G}^{\mu}[\varepsilon\zeta]\psi + \varepsilon\partial_x\zeta\partial_x\psi)^2}{1+\varepsilon^2\mu(\partial_x\zeta)^2}|_{H^{\alpha+1}}\\
       \leq &C(\mu_{\max},|(\partial_x\zeta)^2|_{H^{\alpha+1}})|(\frac{1}{\mu}\mathcal{G}^{\mu}[\varepsilon\zeta]\psi + \varepsilon\partial_x\zeta\partial_x\psi)^2|_{H^{\alpha+1}} \\
       \leq &C(\mu_{\max},|\zeta|_{H^{\alpha+2}})|\frac{1}{\mu}\mathcal{G}^{\mu}[\varepsilon\zeta]\psi + \varepsilon\partial_x\zeta\partial_x\psi|_{H^{\alpha+1}}^2 \\
       \leq &M(\alpha+3)|\psi|_{\dot{H}^{\alpha+3}}.
   \end{align*}
   Then, it only remains to see that using product estimates \ref{product estimate} and $|(\mathrm{F}_{\mu}-1)[u]|_{H^{\alpha}} \lesssim \mu|u|_{H^{\alpha+2}}$ for any $u \in H^{\alpha+2}(\mathbb{R})$, we have
   \begin{align*}
       |(\partial_x\psi)^2 - (\mathrm{F}_{\mu}[\partial_x\psi])^2|_{\dot{H}^{\alpha+1}} \leq &|(\partial_x\psi)^2 - (\mathrm{F}_{\mu}[\partial_x\psi])^2|_{H^{\alpha+1}} \\
       \leq &|\partial_x\psi - \mathrm{F}_{\mu}[\partial_x\psi]|_{H^{\alpha+1}}|\partial_x\psi + \mathrm{F}_{\mu}[\partial_x\psi]|_{H^{\alpha+1}}
       \lesssim \mu |\psi|_{\dot{H}^{\alpha+4}}|\psi|_{\dot{H}^{\alpha+2}}.
   \end{align*}
   This ends the proof.
\end{proof}

\begin{myrem}
    \begin{itemize}
        \item When deriving the Hamilton's equations associated with $H_0 + \varepsilon H_1$ we get the following Whitham-Boussinesq system
        \begin{align*}
            \begin{cases}
                \partial_t \zeta + \mathrm{F}_{\mu}^2[\partial_x^2 \psi] + \varepsilon \mathrm{F}_{\mu}\partial_x[\zeta \mathrm{F}_{\mu}[\partial_x\psi]] = 0, \\
                \partial_t \psi + \zeta + \frac{\varepsilon}{2}(\mathrm{F}_{\mu}[\partial_x \psi])^2 = 0.
            \end{cases}
        \end{align*}
        It is equivalent to the system \eqref{Whitham-Boussinesq system} at a precision order $O(\mu\varepsilon)$.
        \item We could approximate the Hamiltonian $H_{\WW}$ with any other perturbation of $H_0 + \varepsilon H_1$ of order $O(\mu\varepsilon)$. It wouldn't change the reasoning. Our choice is made to simplify the computations. 
    \end{itemize}
\end{myrem}

\begin{mynots}
    We will denote by $\widetilde{H}_0 + \varepsilon \widetilde{H}_1$ the Hamiltonian $H_0 + \varepsilon H_1$ rewritten in the variables $(\zeta,v)$ where $v = \partial_x\psi$. It is associated with the Poisson tensor $\widetilde{J} = \begin{pmatrix} 0 & -\partial_x \\ -\partial_x & 0 \end{pmatrix}$. We recall that a change of unknowns $\mathcal{T}(\eta, \omega)$ turns the initial Poisson tensor, denoted $J$, into $\Big( \frac{\partial F}{\partial(\eta,\omega)} \Big) J \Big( \frac{\partial F}{\partial(\eta,\omega)} \Big)^*$, where $*$ is the adjoint in $L^2(\mathbb{R})^2$.
\end{mynots}
Before applying Birkhoff's algorithm, we change the unknowns of the approximated Hamiltonian $H_0 + \varepsilon H_1$ with those usually used to diagonalize the linear part of the water waves equations.

\begin{mypp}\label{Transformation Diag}
    Let $r$ and $s$ be defined as
\begin{align}\label{Diagonal transformation}
    r \colonequals \frac{\zeta + \mathrm{F}_{\mu}[\partial_x\psi]}{2}, \ \ s \colonequals \frac{\zeta - \mathrm{F}_{\mu}[\partial_x\psi]}{2}.
\end{align}
We have
\begin{align}\label{Hamiltonian in r and s}
    (\widetilde{H}_0 + \varepsilon \widetilde{H}_1)(\zeta,v) = (H_0 + \varepsilon H_1)(\zeta,\psi) = \int_{\mathbb{R}} (r^2 + s^2) \dx + \frac{\varepsilon}{2} \int (r^3 + s^3) \dx -\frac{\varepsilon}{2} \int_{\mathbb{R}} (r^2 s + r s^2) \dx. 
\end{align}
We will denote the later Hamiltonian $H_{\BW}(r,s)$. The Poisson tensor associated with is $J_{\mu} = \begin{pmatrix} -\frac{\mathrm{F}_{\mu} \partial_x}{2} & 0 \\ 0 & \frac{\mathrm{F}_{\mu} \partial_x}{2} \end{pmatrix}$.
\end{mypp}
\begin{proof}
    Easy computations.
\end{proof}

For later purposes we also define the inverse of the transformation \eqref{Diagonal transformation}.
\begin{myppr}\label{From HBW to H0 epsilon H1}
    Let $\alpha \geq 0$ and $\mathcal{T}_{\rm{D}}: H^{\alpha+1}(\mathbb{R})\times H^{\alpha+1}(\mathbb{R}) \to H^{\alpha}(\mathbb{R}) \times H^{\alpha}(\mathbb{R})$ be defined by 
    \begin{align}\label{Transformation Diag WW}
        \mathcal{T}_{\rm{D}}(r,s) \colonequals \begin{pmatrix} r+s \\ \mathrm{F}_{\mu}^{-1}[r-s] \end{pmatrix}.
    \end{align}
    Then $(\widetilde{H}_0 + \varepsilon \widetilde{H}_1)(\mathcal{T}_{\rm{D}}(r,s)) = H_{\BW}(r,s)$.
\end{myppr}
\begin{proof}
      For any $\alpha \geq 0$ and any $u \in H^{\alpha+1}(\mathbb{R})$, we have $|\mathrm{F}_{\mu}^{-1}[u]|_{H^{\alpha}} \lesssim |u|_{H^{\alpha+1}}$. 
\end{proof}

To separate each important part of the Hamiltonian $H_{\BW}$ we set some notations.
\begin{mynots}
Let
\begin{align*}
    \begin{cases}
        L \colonequals \int_{\mathbb{R}} (r^2 + s^2) \dx, \\
        Z \colonequals \frac{1}{2} \int (r^3 + s^3) \dx, \\
        W \colonequals -\frac{1}{2} \int_{\mathbb{R}} (r^2 s + r s^2) \dx.
    \end{cases}
\end{align*}
We have
\begin{align*}
    H_{\BW} = L + \varepsilon Z + \varepsilon W.
\end{align*}
\end{mynots}
\begin{myrem}
    Our goal is to decouple the equations associated with the Hamiltonian $H_{\BW}$. We remark that only $W$ gives coupled terms in the Hamilton's equations. 
\end{myrem}

\label{From the water waves equations to a Whitham-Boussinesq system}

\subsection{Application of Birkhoff's algorithm}
We suppose first that all our objects of study are smooth. If $G(r,s)$ is a smooth function then it's corresponding Hamilton's equations are
\begin{align*}
    \partial_t \begin{pmatrix} r \\ s \end{pmatrix} = J_{\mu} \nabla G.
\end{align*}
We will denote by $\Phi_G^t$ the corresponding flow. 

If $F(r,s)$ is another smooth function, we denote the Lie derivative of $F$ with respect to $G$ by 
\begin{align*}
    \{ F, G \}_{\mu} \colonequals (\nabla F, J_{\mu} \nabla G ),
\end{align*}
where $(\cdot,\cdot)$ denotes the scalar product in $L^2(\mathbb{R})^2$.

\begin{myppr}
    Let $G(r,s)$ be a smooth function and $\Phi_G^t$ its flow associated with the Poisson tensor $J_{\mu}$. Let also $F(r,s)$ be another smooth function. Then 
    \begin{align*}
        F \circ \Phi_{G}^{\varepsilon} = F + \varepsilon \{F, G \}_{\mu} + O(\varepsilon^2). 
    \end{align*}
\end{myppr}
\begin{proof}
    See Subsection 4.1 in \cite{Bambusi20}.
\end{proof}

Suppose for now that in our case the Hamiltonian $H_{\BW}$ is smooth. Then for any smooth function $G(r,s)$ we get
\begin{align*}
    H_{\BW} \circ \Phi_{G}^{\varepsilon} = L + \varepsilon \{L, G\}_{\mu} + \varepsilon Z + \varepsilon W + O(\varepsilon^2).
\end{align*}

\begin{myrem}
Here it seems that to get a normal form at the order of precision $O(\varepsilon^2)$, we need $G$ to solve the homological equation
\begin{align}\label{Homological equation}
    \{ L, G\}_{\mu} + W = 0.
\end{align}
In \cite{Bambusi20}, the author gave the solution of such equations. If 
\begin{align}\label{Assumption for a solution of the homological equation}
    \lim_{\tau \to +\infty} W(\Phi_{L}^{\tau}) + W(\Phi_{L}^{-\tau}) = 0,
\end{align}
(in our case, we can prove this assumption using a Littlewood-Paley decomposition)
then the solution of \eqref{Homological equation} is
\begin{align}\label{Solution of the Homological solution with the exact Poisson tensor}
    G(r,s) = -\frac{1}{2} \int_{\mathbb{R}} \sgn(\tau) W(\Phi_{L}^{\tau}(r,s)) d\tau,
\end{align}
at the additional condition that the above function is well-defined.

However the latter condition is not easy to verify in general. We expose here a naive attempt to do so in our context. 

The flow $\Phi_{L}^\tau$ satisfy
\begin{align*}
    \frac{d}{d\tau} \Phi_{L}^\tau = J_{\mu} \nabla L (\Phi_{L}^\tau).
\end{align*}
So that
\begin{align*}
    \Phi_{L}^\tau (r,s) = \begin{pmatrix}e^{-iD\mathrm{F}_{\mu}\tau}r \\ e^{iD\mathrm{F}_{\mu}\tau}s \end{pmatrix},
\end{align*}
and 
\begin{align*}
    W(\Phi_{L}^{\tau}(r,s)) = -\frac{1}{2} \int_{\mathbb{R}} (e^{-iD\mathrm{F}_{\mu}\tau}r)^2 e^{iD\mathrm{F}_{\mu}\tau}s \dx - \frac{1}{2} \int_{\mathbb{R}} (e^{iD\mathrm{F}_{\mu}\tau}s)^2 e^{-iD\mathrm{F}_{\mu}\tau}r \dx. 
\end{align*}
So 
\begin{align*}
    |W(\Phi_{L}^{\tau}(r,s))| &\leq \int_{\mathbb{R}} |e^{-iD\mathrm{F}_{\mu}\tau}r|^2 |e^{iD\mathrm{F}_{\mu}\tau}s| \dx + \int_{\mathbb{R}} |e^{iD\mathrm{F}_{\mu}\tau}s|^2 |e^{-iD\mathrm{F}_{\mu}\tau}r| \dx\\
    &\leq |e^{iD\mathrm{F}_{\mu}\tau}s|_{\infty} |e^{-iD\mathrm{F}_{\mu}\tau}r|_2^2 + |e^{-iD\mathrm{F}_{\mu}\tau}r|_{\infty} |e^{iD\mathrm{F}_{\mu}\tau}s|_2^2.
\end{align*}
But $|e^{-iD\mathrm{F}_{\mu}\tau}r|_2 = |r|_2$, and the dispersive estimates only give a decrease in time of $|e^{iD\mathrm{F}_{\mu}\tau}s|_{\infty}$ and $|e^{-iD\mathrm{F}_{\mu}\tau}r|_{\infty}$ of order $1/t^{-1/2}$ \cite{Bulut2016}.

The idea to overcome this issue is to only solve an approximation of the homological equation at order $O(\mu\varepsilon)$. 
\end{myrem}

We define another Lie derivative, associated with the Poisson tensor $J_{\simp} \colonequals \begin{pmatrix} -\partial_x/2 & 0 \\ 0 & \partial_x/2 \end{pmatrix}$: 
\begin{align*}
    \{ L, G\}_{\simp} \colonequals (\nabla L, J_{\simp} \nabla G)
\end{align*}
\begin{myppr}
    We have 
    \begin{align*}
        \{ L, G\}_{\mu} = \{ L, G\}_{\simp} + O(\mu)
    \end{align*}
\end{myppr}
\begin{proof}
    By definition $\{ L, G\}_{\mu} = (\nabla L, J_{\mu} \nabla G)$, with $J_{\mu} = \begin{pmatrix} -\frac{\mathrm{F}_{\mu} \partial_x}{2} & 0 \\ 0 & \frac{\mathrm{F}_{\mu} \partial_x}{2} \end{pmatrix}$. So 
    \begin{align*}
        \{ L, G\}_{\mu} - \{ L, G\}_{\simp} = ((\nabla L, (J_{\mu}-J_{\simp}) \nabla G),
    \end{align*}
    with $J_{\mu}-J_{\simp} = \begin{pmatrix} - \frac{\mathrm{F}_{\mu}-1}{2} & 0 \\ 0 & \frac{\mathrm{F}_{\mu}-1}{2} \end{pmatrix}$.
    But for any $\alpha \geq 0$ and any $u \in H^{\alpha+2}(\mathbb{R})$, we have $|(\mathrm{F}_{\mu} - 1)[u]|_{H^\alpha} \lesssim \mu |u|_{H^{\alpha+2}}$. 
\end{proof}

The Lie derivative associated with the Poisson tensor $J_{\simp}$ is the same one as in Bambusi's article. To write an explicit expression of the solution of the simplified homological equation associated with 
\begin{align}\label{Simplified homological equation}
        \{L, G \}_{\simp} + W =0,
\end{align}
we need the definition of a classical primitive operator and some properties on it.

\begin{mydefpp}\label{primitive operator}
Let $\alpha \geq 0$. We define the operator $\partial^{-1}: W^{\alpha,1}(\mathbb{R}) \to W^{\alpha,\infty}(\mathbb{R})$ by the formula
\begin{align}\label{partial -1}
    \partial^{-1} u(y) = \frac{1}{2} \int_{\mathbb{R}} \sgn(y-y_1) u(y_1) \dy_1.
\end{align}
Here are some properties on it:
\begin{itemize}
    \item It is continuous.
    \item It is skew-adjoint for the scalar product in $L^2(\mathbb{R})$.
    \item If $\lim_{x \to +\infty} u(x) + u(-x) = 0$ (it is the case for any function $u \in W^{1,1}(\mathbb{R})$, see Corollary 8.9 in \cite{Brezis}), then $\partial^{-1} (\partial_x u) = u$.
    \item $\partial_x (\partial^{-1} u)= u$.
\end{itemize}
\end{mydefpp}

We can now write explicitly the solution of the simplified homological equation.
\begin{mypp}\label{Solution of the simplified homological equation}
    Let $W(r,s) = -\frac{1}{2} \int_{\mathbb{R}} (r^2 s + r s^2) \dx$, and $L(r,s) = \int_{\mathbb{R}} (r^2 + s^2) \dx$ for any $(r,s) \in W^{1,1}(\mathbb{R})$. Then the solution of the simplified homological equation \eqref{Simplified homological equation} is 
    \begin{align*}
        G(r,s) = \frac{1}{4} \int_{\mathbb{R}} [\partial^{-1} (r^2) s + \partial^{-1}(r) s^2]\dx.
    \end{align*}
\end{mypp}
\begin{proof}
    See Lemma 5.2 in \cite{Bambusi20}.
\end{proof}

Here $G$ is clearly well-defined for $r$ and $s$ in $W^{\alpha,1}(\mathbb{R})$ with $\alpha \geq 1$ large enough. However there is another problem, $G$ does not generate a flow. For that reason we use the generalization of Birkhoff's algorithm in the case of "almost smooth" functions depending on small parameters (see \cite{Bambusi20}). We recall here what is an almost smooth function.
\begin{mydef}\label{almost smooth functions}
    Let $\{\mathcal{B}^{\alpha}\}_{\alpha \geq 0}$ be a Banach scale. A map $F(r,s,\mu,\varepsilon)$ is said to be almost smooth if $\forall \beta, \gamma \geq 0$, there exists $\delta$ and an open neighbourhood of the origin $\mathcal{U} \subset \mathcal{B}^{\delta}\times \mathbb{R} \times \mathbb{R}$ (where $\mathbb{R}\times \mathbb{R}$ is the domain of $(\mu,\varepsilon)$), such that
    \begin{align*}
        F \in \mathcal{C}^{\beta}(\mathcal{U}, \mathcal{B}^{\gamma}\times\mathbb{R}\times \mathbb{R}).
    \end{align*}
\end{mydef}

Due to the definition of $\partial^{-1}$, for the rest of the article, we will work with the Banach scale $\{W^{\alpha,1}(\mathbb{R}) \times W^{\alpha,1}(\mathbb{R}) \}_{\alpha \geq 0}$.
If 
\begin{itemize}
    \item for any $\alpha \geq 0$, there exists $\alpha' \geq 0$ such that the Poisson tensor $J_{\mu}$ is bounded from $W^{\alpha',1}(\mathbb{R})\times W^{\alpha',1}(\mathbb{R})$ to $W^{\alpha,1}(\mathbb{R})\times W^{\alpha,1}(\mathbb{R})$,
    \item the function $G(r,s)$ defined in Proposition \ref{Solution of the Homological solution with the exact Poisson tensor} has an almost smooth vector field, i.e. $J_{\mu}\nabla G$ is almost smooth, same for $\{ L, G\}_{\simp}$,  
\end{itemize} 
then the theory developed in \cite{Bambusi20} gives the existence of a transformation $\mathcal{T}_{\rm{B}}$ such that 
\begin{align*}
    H_{\BW} \circ \mathcal{T}_{\rm{B}} &= L + \varepsilon \{L, G\}_{\mu} + \varepsilon Z + \varepsilon W + O(\varepsilon^2) \\
    &= L + \varepsilon \{L, G\}_{\simp} + \varepsilon Z + \varepsilon W + O(\mu\varepsilon + \varepsilon^2)\\
    &= L + \varepsilon Z + O(\mu\varepsilon + \varepsilon^2),
\end{align*}
where $A = B + O(\mu^k \varepsilon^l)$ now means $\frac{A-B}{\mu^k \varepsilon^l}$ has an almost smooth vector field. 

We verify the latter two conditions.
\begin{myppr}\label{Boundedness of the Poisson tensor}
    For any $\alpha\geq 0$ the Fourier multiplier $\mathrm{F}_{\mu}$ is bounded from $W^{\alpha+1,1}(\mathbb{R})$ to $W^{\alpha,1}(\mathbb{R})$. 
    In particular, the Poisson tensor $J_{\mu}$ is bounded from $W^{\alpha+2,1}(\mathbb{R})\times W^{\alpha+2,1}(\mathbb{R})$ to $W^{\alpha,1}(\mathbb{R})\times W^{\alpha,1}(\mathbb{R})$.
    
    The inverse Fourier multiplier of $\mathrm{F}_{\mu}$, denoted $\mathrm{F}_{\mu}^{-1} \colonequals \sqrt{\frac{\sqrt{\mu}|D|}{\tanh{(\sqrt{\mu}|D|)}}}$, is bounded from $W^{\alpha+2,1}(\mathbb{R})$ to $W^{\alpha,1}(\mathbb{R})$.
\end{myppr}
\begin{proof}
    First, we remark that we only need to prove the boundedness of $\mathrm{F}_{\mu}$ in $ W^{1,1}(\mathbb{R}) \to L^1(\mathbb{R})$. To do so, we consider an even cutoff function $\chi \in \mathcal{C}^{\infty}((-2,2))$ such that $\chi(\xi) = 1$ for $|\xi| \leq 1$. Let $f$ be in $W^{1,1}(\mathbb{R})$, we can decompose $\mathrm{F}_{\mu}[f]$ into
    \begin{align*}
        \mathrm{F}_{\mu}[f](x) = &\int_{\mathbb{R}} \int_{\mathbb{R}} e^{-i(y-x)\xi} F_{\mu}(\xi) \chi(\xi)f(y) \dy d\xi + \int_{\mathbb{R}} \int_{\mathbb{R}} e^{-i(y-x)\xi} F_{\mu}(\xi) (1-\chi(\xi))f(y) \dy d\xi \\
        \colonequals &I_1 + I_2.
    \end{align*}
    We deal with each term separately. 
    \begin{itemize}
        \item We begin with $I_1$. We can write it using the inverse Fourier transform, denoted here $\mathcal{F}^{-1}$:
        \begin{align*}
            I_1 = \int_{\mathbb{R}} \int_{\mathbb{R}} e^{-i(y-x)\xi} F_{\mu}(\xi) \chi(\xi)f(y) \dy d\xi = \mathcal{F}^{-1}[F_{\mu}(\xi)\chi(\xi)]* f.
        \end{align*}
        Using Young's inequality, we get
        \begin{align*}
            |\mathcal{F}^{-1}[F_{\mu}(\xi)\chi(\xi)]* f|_{L^1} \leq |\mathcal{F}^{-1}[F_{\mu}(\xi)\chi(\xi)]|_{L^1} |f|_{L^1}.
        \end{align*}
        But
        \begin{align*}
            |\mathcal{F}^{-1}[F_{\mu}(\xi)\chi(\xi)]|_{L^1} \lesssim |F_{\mu}(\xi)\chi(\xi))|_{H^1}.
        \end{align*}
        And $\forall \xi \in \mathbb{R}$ $|F_{\mu}(\xi)|, |F_{\mu}'(\xi)| \leq 1$, so $\mathcal{F}^{-1}[F_{\mu}(\xi)\chi(\xi)]$ is in $L^1(\mathbb{R})$.
        It finishes the case of $I_1$.
        \item Now we deal with $I_2$. Doing an integration by parts, we obtain 
        \begin{align*}
            \int_{\mathbb{R}} \int_{\mathbb{R}} e^{-i(y-x)\xi} F_{\mu}(\xi) (1-\chi(\xi))f(y) \dy d\xi &= \int_{\mathbb{R}}\int_{\mathbb{R}} e^{-i(y-x)\xi} \frac{F_{\mu}(\xi)}{i\xi} (1-\chi(\xi)) \partial_y f(y) \dy d\xi \\
            &= \mathcal{F}^{-1}[\frac{F_{\mu}(\xi)}{i\xi}(1-\chi(\xi))]*\partial_y f
        \end{align*}
        Using again Young's inequality, we get 
        \begin{align*}
            |\mathcal{F}^{-1}[\frac{F_{\mu}(\xi)}{i\xi}(1-\chi(\xi))]*\partial_y f|_{L^1} \leq |\mathcal{F}^{-1}[\frac{F_{\mu}(\xi)}{i\xi}(1-\chi(\xi))]|_{L^1} |\partial_y f|_{L^1}.
        \end{align*}
        But
        \begin{align*}
            |\mathcal{F}^{-1}[\frac{F_{\mu}(\xi)}{i\xi}(1-\chi(\xi))]|_{L^1} \lesssim |\frac{F_{\mu}(\xi)}{i\xi}(1-\chi(\xi))|_{H^1}.
        \end{align*}
        So $\mathcal{F}^{-1}[\frac{F_{\mu}(\xi)}{i\xi}(1-\chi(\xi))] \in L^1(\mathbb{R})$. It finishes the case of $I_2$.
    \end{itemize}
    At the end, we have $|\mathrm{F}_{\mu}[f]|_{L^1} \leq |I_1|_{L^1} + |I_2|_{L^1} \lesssim |f|_{W^{1,1}}$.
    
    With the same process we obtain the result on $\mathrm{F}_{\mu}^{-1}$ (we need a second integration by parts in the high frequencies to get $\mathcal{F}^{-1}[\frac{F_{\mu}^{-1}(\xi)}{(i\xi)^2}(1-\chi(\xi))]$ in $L^1$). 
\end{proof}

\begin{myppr}
    The function $G(r,s)$ defined in Proposition \ref{Solution of the simplified homological equation} and $\{L, G\}_{\simp}$ have an almost smooth vector field.
\end{myppr}
\begin{proof}
    First we have $2 \partial_r G = -r\partial^{-1} (s) - \frac{1}{2}\partial^{-1}(s^2) $, where we used the skew-adjoint property of $\partial^{-1}$. So $-2 \mathrm{F}_{\mu} \partial_x \partial_r G = \mathrm{F}_{\mu}[\partial_x (r) \partial^{-1}(s)] + \mathrm{F}_{\mu}[rs] + \frac{1}{2}\mathrm{F}_{\mu}[s^2]$. But for any $\alpha \geq 0$, $\mathrm{F}_{\mu}$ is bounded from $W^{\alpha+1,1}(\mathbb{R})$ to $W^{\alpha,1}(\mathbb{R})$, and $W^{\alpha+1,1}(\mathbb{R})$ is an algebra. We also know that $\partial^{-1}$ is a linear operator bounded from $W^{\alpha+1,1}(\mathbb{R})$ to $W^{\alpha+1,\infty}(\mathbb{R})$ , and the product of a $W^{\alpha+1,\infty}(\mathbb{R})$ with a $W^{\alpha+1,1}(\mathbb{R})$ function is still in $W^{\alpha+1,1}(\mathbb{R})$. Combining all these elements, we get the result for $G$.
   
   The same arguments give the result for $\{ L, G\}_{\simp}$.
\end{proof}

We have an explicit formulation of the transformation $\mathcal{T}_{\rm{B}}$ given by (4.25) in \cite{Bambusi20}, taking $X = J_{\simp} \nabla G$, on which we can do precise estimates.
\begin{mypp}\label{proposition explicit formulation of T}
    Let $\alpha \geq 0 $ and $\mathcal{T}_{\rm{B}}: W^{\alpha+1,1}(\mathbb{R})\times W^{\alpha+1,1}(\mathbb{R}) \to W^{\alpha,1}(\mathbb{R})\times W^{\alpha,1}(\mathbb{R})$ be defined by
    \begin{align}\label{explicit formulation of T}
        \mathcal{T}_{\rm{B}}(r,s) = \begin{pmatrix} r \\ s \end{pmatrix} + \varepsilon J_{\simp} \nabla G(r,s) = \begin{pmatrix} r + \frac{\varepsilon}{4} \partial_x (r) \partial^{-1}(s) + \frac{\varepsilon}{4} rs + \frac{\varepsilon}{8}s^2 \\ s + \frac{\varepsilon}{4} \partial_x (s) \partial^{-1}(r) + \frac{\varepsilon}{4} rs + \frac{\varepsilon}{8}r^2 \end{pmatrix}.
    \end{align}
    Then 
    \begin{align*}
        H_{\BW} \circ \mathcal{T}_{\rm{B}} = L + \varepsilon Z + O(\mu\varepsilon + \varepsilon^2)
    \end{align*}
\end{mypp}
\begin{proof}
    Let $r,s \in W^{\alpha+1,1}(\mathbb{R})$. The latter Sobolev space in an algebra, so the only terms we need to look at are $\partial_x r \partial^{-1}s$ and $\partial_x s \partial^{-1}r$. We have
    \begin{align*}
        |\partial_x r \partial^{-1}s|_{W^{\alpha,1}} \leq |\partial^{-1}s|_{W^{\alpha,\infty}}|\partial_x r|_{W^{\alpha,1}}.
    \end{align*}
    Here using the continuity of $\partial^{-1}$ from $L^1(\mathbb{R})$ to $L^{\infty}(\mathbb{R})$ for any $\beta \geq 0$ and the third property of Definition/Property \ref{primitive operator}, we get 
    \begin{align}\label{estimates in Sobolev infinity}
        \begin{cases}
            |\partial^{-1}s|_{W^{\alpha,\infty}} \lesssim |s|_{L^1} \ \ \If \ \ \alpha = 0,\\
            |\partial^{-1}s|_{W^{\alpha,\infty}} \lesssim |s|_{L^1} + \sum\limits_{\beta \leq \alpha} |\partial_x^{\beta} \partial^{-1} s|_{L^{\infty}} = |s|_{L^1} + \sum\limits_{\beta \leq \alpha} |\partial^{-1} \partial_x^{\beta} s|_{L^{\infty}} \lesssim |s|_{W^{\alpha,1}} \ \ \If \ \ \alpha \geq 1.
        \end{cases}
    \end{align}
    Idem for the other term. Thus $\mathcal{T}_{\rm{B}}: W^{\alpha+1,1}(\mathbb{R})\times W^{\alpha+1,1}(\mathbb{R}) \to W^{\alpha,1}(\mathbb{R})\times W^{\alpha,1}(\mathbb{R})$.
    
    From the definition of $H_{\BW}$ (see \eqref{Hamiltonian in r and s}) and \eqref{explicit formulation of T} we have
    \begin{align*}
        H_{\BW}(\mathcal{T}_{\rm{B}}(r,s)) = &\int_{\mathbb{R}} \Big( r^2 + \frac{\varepsilon}{2} r\partial_x r \partial^{-1}s + \frac{\varepsilon}{2} r^2s + \frac{\varepsilon}{4}rs^2 \Big) \dx \\
        + &\int_{\mathbb{R}} \Big(  s^2 + \frac{\varepsilon}{2} s\partial_x s \partial^{-1}r + \frac{\varepsilon}{2} rs^2 + \frac{\varepsilon}{4}s r^2 \Big) \dx \\
        - &\frac{\varepsilon}{2}\int_{\mathbb{R}} (r^2 s + r s^2) \dx + \frac{\varepsilon}{2}\int_{\mathbb{R}} (r^3 + s^3) \dx + O(\varepsilon^2). 
    \end{align*}
    But using the skew-adjointness of $\partial^{-1}$ we get
    \begin{align*}
        \begin{cases}
            \int_{\mathbb{R}} r\partial_x r \partial^{-1}s \dx = \frac{1}{2} \int_{\mathbb{R}} \partial_x(r^2) \partial^{-1} s \dx = - \frac{1}{2} \int_{\mathbb{R}} \partial^{-1}  (\partial_x r^2) s \dx = - \frac{1}{2} \int_{\mathbb{R}}  r^2 s \dx, \\
            \int_{\mathbb{R}} s\partial_x s \partial^{-1}r \dx = \frac{1}{2} \int_{\mathbb{R}} \partial_x(s^2) \partial^{-1} r \dx = - \frac{1}{2} \int_{\mathbb{R}} \partial^{-1}  (\partial_x s^2) r \dx = - \frac{1}{2} \int_{\mathbb{R}}  s^2 r \dx.
        \end{cases}
    \end{align*}
    Thus 
    \begin{align*}
        H_{\BW}(\mathcal{T}_{\rm{B}}(r,s)) &= \int_{\mathbb{R}} (r^2 + s^2) \dx + \frac{\varepsilon}{2} \int_{\mathbb{R}} (r^3 + s^3) \dx + O(\varepsilon^2) \\
        &= L(r,s) + \varepsilon Z(r,s) + O(\varepsilon^2).
    \end{align*}
\end{proof}

\begin{myrem}
    The theory developped in \cite{Bambusi20} tells us that the transformation $\mathcal{T}_{\rm{B}}$ should be $\mathcal{T}_{\rm{B}}(r,s) = \begin{pmatrix} r \\ s \end{pmatrix} + \varepsilon J_{\mu} \nabla G(r,s)$. But the precision we want is $O(\mu\varepsilon + \varepsilon^2)$, and $\varepsilon J_{\mu} = \varepsilon J_{\simp} + O(\mu\varepsilon)$.
\end{myrem}

The transformation \eqref{explicit formulation of T} preserves the Hamiltonian structure associated with the Poisson tensor $J_{\mu}$ at the order of precision $O(\mu\varepsilon + \varepsilon^2)$, see the following proposition.

\begin{mypp}\label{preservation of the Hamiltonian structure}
    Let $\alpha \geq 0$. Let $r,s \in W^{\alpha+4,1}(\mathbb{R})$ Let $\frac{\partial \mathcal{T}_{\rm{B}}}{\partial(r,s)}$ be the Jacobian matrix of the transformation $T$. We have 
    \begin{align*}
        \Big( \frac{\partial \mathcal{T}_{\rm{B}}}{\partial(r,s)} \Big) J_{\mu} \Big(\frac{\partial \mathcal{T}_{\rm{B}}}{\partial(r,s)} \Big)^* = J_{\mu} + (\mu\varepsilon + \varepsilon^2) R, 
    \end{align*}
    where $*$ is the adjoint in $L^2(\mathbb{R})^2$ and $R$ is linear operator such that for any $U \in H^{\alpha+4}(\mathbb{R})\times H^{\alpha+4}(\mathbb{R})$, $|RU|_{H^{\alpha}\times H^{\alpha}} \leq C(\mu_{\max}, |r|_{W^{\alpha+4,1}},|s|_{W^{\alpha+4,1}})|U|_{H^{\alpha+4}\times H^{\alpha+4}}$.
\end{mypp}
\begin{proof}
    We easily compute the Jacobian matrix of $\mathcal{T}_{\rm{B}}$:
    \begin{align*}
        \frac{\partial \mathcal{T}_{\rm{B}}}{\partial(r,s)} = \begin{pmatrix} 1 + \frac{\varepsilon}{4} s + \frac{\varepsilon}{4} \partial^{-1}(s) \partial_x(\circ) & \frac{\varepsilon}{4} \partial_x (r) \partial^{-1}(\circ) + \frac{\varepsilon}{4}(r+s) \\ \frac{\varepsilon}{4} \partial_x (s) \partial^{-1}(\circ) + \frac{\varepsilon}{4}(r+s) & 1 + \frac{\varepsilon}{4} r + \frac{\varepsilon}{4} \partial^{-1}(r) \partial_x(\circ)\end{pmatrix}.
    \end{align*}
    Its $L^2$ adjoint is
    \begin{align*}
        \Big( \frac{\partial \mathcal{T}_{\rm{B}}}{\partial(r,s)} \Big)^* = \begin{pmatrix} 1 + \frac{\varepsilon}{4}s - \frac{\varepsilon}{4}\partial_x(\partial^{-1}(s) \circ) & \frac{\varepsilon}{4} (r+s) -\frac{\varepsilon}{4} \partial^{-1}(\partial_x (s) \circ) \\ \frac{\varepsilon}{4} (r+s) -\frac{\varepsilon}{4} \partial^{-1}(\partial_x (r) \circ) & 1 + \frac{\varepsilon}{4}r - \frac{\varepsilon}{4}\partial_x(\partial^{-1}(r) \circ).
        \end{pmatrix}
    \end{align*}
    Using the fact that both $\frac{\partial \mathcal{T}_{\rm{B}}}{\partial(r,s)}$ and $\Big(\frac{\partial \mathcal{T}_{\rm{B}}}{\partial(r,s)} \Big)^*$ can be written under the form $I_d + O(\varepsilon)$, we write 
    \begin{align*}
        \Big( \frac{\partial \mathcal{T}_{\rm{B}}}{\partial(r,s)} \Big) J_{\mu} \Big(\frac{\partial \mathcal{T}_{\rm{B}}}{\partial(r,s)} \Big)^* &= \Big[\Big( \frac{\partial \mathcal{T}_{\rm{B}}}{\partial(r,s)} - I_d\Big) + I_d \Big] J_{\mu} \Big[ \Big(\Big(\frac{\partial \mathcal{T}_{\rm{B}}}{\partial(r,s)} \Big)^* - I_d \Big) + I_d \Big] \\
        &= \Big( \frac{\partial \mathcal{T}_{\rm{B}}}{\partial(r,s)} - I_d \Big) J_{\mu} \Big(\Big(\frac{\partial \mathcal{T}_{\rm{B}}}{\partial(r,s)} \Big)^* - I_d \Big) \\
        &+ \Big( \frac{\partial \mathcal{T}_{\rm{B}}}{\partial(r,s)} -I_d \Big) J_{\mu} + J_{\mu} \Big(\Big(\frac{\partial \mathcal{T}_{\rm{B}}}{\partial(r,s)} \Big)^* - I_d \Big) + J_{\mu}.  
    \end{align*}
    The first term is of order $O(\varepsilon^2)$, we will estimate it later. The second and third terms are of order $O(\varepsilon)$. We prove now that their addition is of order $O(\mu\varepsilon)$. We easily compute both terms:
    \begin{align*}
        \Big( \frac{\partial \mathcal{T}_{\rm{B}}}{\partial(r,s)} -I_d \Big) J_{\mu} = \frac{\varepsilon}{8} \begin{pmatrix} -s \mathrm{F}_{\mu}\partial_x[\circ] - \partial^{-1}(s) \mathrm{F}_{\mu}\partial_x^2[\circ] & (r+s)\mathrm{F}_{\mu}\partial_x[\circ] + \partial_x (r) \mathrm{F}_{\mu}[\circ] \\  - (r+s)\mathrm{F}_{\mu}\partial_x[\circ] -\partial_x (s) \mathrm{F}_{\mu}[\circ] & r \mathrm{F}_{\mu}\partial_x[\circ] + \partial^{-1}(r) \mathrm{F}_{\mu}\partial_x^2[\circ] \end{pmatrix},
    \end{align*}
    and 
    \begin{align*}
        J_{\mu} \Big(\Big(\frac{\partial \mathcal{T}_{\rm{B}}}{\partial(r,s)} \Big)^* - I_d \Big) = \frac{\varepsilon}{8} \begin{pmatrix} \mathrm{F}_{\mu}\partial_x [\partial^{-1}(s)\partial_x(\circ)] & -\mathrm{F}_{\mu}\partial_x[(r+s)\circ] + \mathrm{F}_{\mu}[\partial_x (s) \circ] \\ \mathrm{F}_{\mu}\partial_x[(r+s)\circ] - \mathrm{F}_{\mu}[\partial_x (r) \circ] & -\mathrm{F}_{\mu}\partial_x [\partial^{-1}(r)\partial_x(\circ)] \end{pmatrix}.
    \end{align*}
    For now, we look at the first term of both matrices. For any $u \in H^{\alpha+4}(\mathbb{R})$
    \begin{align*}
        &|\mathrm{F}_{\mu}\partial_x [\partial^{-1}(s)\partial_x u] -s \mathrm{F}_{\mu}\partial_x[u] - \partial^{-1}(s) \mathrm{F}_{\mu}\partial_x^2[u]|_{H^{\alpha}} \\
        \leq & |[\mathrm{F}_{\mu},s]\partial_x u|_{H^{\alpha}} + |[\mathrm{F}_{\mu},\partial^{-1}s]\partial_x^2 u|_{H^{\alpha}} \\
        = & |[\mathrm{F}_{\mu}-1;s]\partial_x u|_{H^{\alpha}} + |[\mathrm{F}_{\mu}-1,\partial^{-1}s]\partial_x^2 u|_{H^{\alpha}}.
    \end{align*}
    But using Remark \ref{remark commutator estimates} and \eqref{estimates in Sobolev infinity} we have 
    \begin{align*}
        |[\mathrm{F}_{\mu}-1,\partial^{-1}s]\partial_x^2 u|_{H^{\alpha}}
        &\lesssim \mu |\partial^{-1}(s) \partial_x^2 u|_{H^{\alpha+2}} + |\partial^{-1}s|_{W^{\alpha+2,\infty}}|(\mathrm{F}_{\mu}-1)[\partial_x^2 u|_{H^{\alpha}}\\
        &\lesssim \mu |\partial^{-1}s|_{W^{\alpha+2,\infty}} |\partial_x^2 u|_{H^{\alpha+2}}\\
        &\lesssim \mu |s|_{W^{\alpha+2,1}}|u|_{H^{\alpha+4}}.
    \end{align*}
    Using algebra properties of $H^{\alpha+2}(\mathbb{R})$ and the Sobolev embedding $W^{\alpha+3,1}(\mathbb{R}) \subset H^{\alpha+2}(\mathbb{R})$, we also have
    \begin{align*}
        |[\mathrm{F}_{\mu}-1,s]\partial_x u|_{H^{\alpha}} \lesssim \mu |s|_{H^{\alpha+2}}|u|_{H^{\alpha+3}} \lesssim \mu |s|_{W^{\alpha+3,1}}|u|_{H^{\alpha+3}}.
    \end{align*}
    So that 
    \begin{align*}
        |\mathrm{F}_{\mu}\partial_x [\partial^{-1}(s)\partial_x u] -s \mathrm{F}_{\mu}\partial_x[u] - \partial^{-1}(s) \mathrm{F}_{\mu}\partial_x^2[u]|_{H^{\alpha}} \lesssim \mu |s|_{W^{\alpha+3,1}}|u|_{H^{\alpha+4}}.
    \end{align*}
    The same can be done for the addition of the fourth term of both matrices.
    
    Now we look at the addition of the second term of both matrices. For any $u \in H^{\alpha+3}(\mathbb{R})$ we have
    \begin{align*}
        (r+s)\mathrm{F}_{\mu}[\partial_x u] + \partial_x (r) \mathrm{F}_{\mu}[u] &= \partial_x((r+s)\mathrm{F}_{\mu}[u]) - \partial_x(r+s)\mathrm{F}_{\mu}[u] \\
        &= \partial_x((r+s)\mathrm{F}_{\mu}[u]) - \partial_x (s) \mathrm{F}_{\mu}[u].
    \end{align*}
    Using Lemma \ref{commutator estimates} and the Sobolev embedding $W^{\alpha+4,1}(\mathbb{R}) \subset H^{\alpha+3}(\mathbb{R})$, we get
    \begin{align*}
        &|\partial_x((r+s)\mathrm{F}_{\mu}[u]) - \partial_x (s) \mathrm{F}_{\mu}[u] - \mathrm{F}_{\mu}[\partial_x((r+s) u)] + \mathrm{F}_{\mu}[\partial_x (s) u]|_{H^{\alpha}}\\
        \leq &|\partial_x([\mathrm{F}_{\mu},r+s]u)|_{H^{\alpha}} + |[\mathrm{F}_{\mu},\partial_x s]u|_{H^{\alpha}} \\
        \leq &|[\mathrm{F}_{\mu}-1,r+s]u|_{H^{\alpha+1}} + |[\mathrm{F}_{\mu}-1,\partial_x s]u|_{H^{\alpha}}\\
        \lesssim &\mu (|r|_{H^{\alpha+3}}+|s|_{H^{\alpha+3}})|u|_{H^{\alpha+3}}\\
        \lesssim &\mu (|r|_{W^{\alpha+4,1}} + |s|_{W^{\alpha+4}})|u|_{H^{\alpha+3}}.
    \end{align*}
    The same can be done for the addition of the third term of both matrices. Thus we get for any $U \in H^{\alpha+4}(\mathbb{R})\times H^{\alpha+4}(\mathbb{R})$
    \begin{align*}
         |\Big( \frac{\partial \mathcal{T}_{\rm{B}}}{\partial(r,s)} -I_d \Big) J_{\mu}U + J_{\mu} \Big(\Big(\frac{\partial \mathcal{T}_{\rm{B}}}{\partial(r,s)} \Big)^* - I_d \Big)U|_{H^{\alpha}} \lesssim \mu \varepsilon (|r|_{W^{\alpha+4,1}}+|s|_{W^{\alpha+4,1}})|U|_{H^{\alpha+4}}
    \end{align*}
    
    It remains to estimate $ \Big( \frac{\partial \mathcal{T}_{\rm{B}}}{\partial(r,s)} - I_d \Big) J_{\mu} \Big(\Big(\frac{\partial \mathcal{T}_{\rm{B}}}{\partial(r,s)} \Big)^*-I_d\Big)$. Using the boundedness of $\mathrm{F}_{\mu}$ and the same tools as before, we get for any $U \in H^{\alpha+3}(\mathbb{R})\times H^{\alpha+3}(\mathbb{R})$
    \begin{align*}
        |\Big( \frac{\partial \mathcal{T}_{\rm{B}}}{\partial(r,s)} - I_d \Big) J_{\mu} \Big(\Big(\frac{\partial \mathcal{T}_{\rm{B}}}{\partial(r,s)} \Big)^*-I_d\Big)U|_{H^{\alpha}} \lesssim \varepsilon^2 C(|r|_{W^{\alpha+4,1}},|r|_{W^{\alpha+4,1}})|U|_{H^{\alpha+3}}.
    \end{align*}
\end{proof}

\label{Application of Birkhoff's algorithm}

\subsection{From the Hamiltonian of the Whitham-Boussinesq system under normal form to two decoupled Whitham equations}

In the previous subsections, we proved the existence of a transformation $\mathcal{T}_{\rm{B}}$ such that 
\begin{align*}
    H_{\WW} \circ \mathcal{T}_{\rm{B}} &= L + \varepsilon Z + O(\mu\varepsilon + \varepsilon^2)\\
    &= \int_{\mathbb{R}} (r^2 + s^2) \dx + \varepsilon \frac{1}{2} \int (r^3 + s^3) \dx + O(\mu\varepsilon + \varepsilon^2).
\end{align*}

The equations associated with the Hamiltonian $L+\varepsilon Z$ and the Poisson tensor $J_{\mu}$ are easily computed.

\begin{myppr}\label{property hamilton's equations HWh}
    The Hamilton's equations associated with the normal form $L + \varepsilon Z$ are
    \begin{align}\label{Decoupled Whitham equations}
        \begin{cases}
            \partial_t r + \mathrm{F}_{\mu}[\partial_x r] + \frac{3\varepsilon}{2} \mathrm{F}_{\mu}[r\partial_x r] = 0, \\
            \partial_t s -\mathrm{F}_{\mu}[\partial_x s] - \frac{3\varepsilon}{2} \mathrm{F}_{\mu}[s \partial_x s] = 0.
        \end{cases}
    \end{align}
\end{myppr}
\begin{proof}
    The Hamilton's equations are 
    \begin{align*}
        \partial_t \begin{pmatrix} r \\ s \end{pmatrix} = J_{\mu} \nabla (L + \varepsilon Z).
    \end{align*}
\end{proof}
\begin{mynot}\label{Notation Hamiltonian HWh}
    We write $H_{\Wh}(r,s) \colonequals L(r,s) + \varepsilon Z(r,s)$.
\end{mynot}
\begin{myrem}\label{Remark two decoupled Whitham equations}
    The two equations \eqref{Decoupled Whitham equations} are almost two Whitham equations. One just need to use that for any $\alpha \geq 0$ and any $u \in H^{\alpha+2}(\mathbb{R})$ one has $|(\mathrm{F}_{\mu} - 1)[u]|_{H^\alpha} \lesssim \mu |u|_{H^{\alpha+2}}$ to get the existence of a remainder $(R_1, R_2)$ controllable in Sobolev spaces $H^{\alpha}$ such that
    \begin{align*}
        \begin{cases}
            \partial_t r + \mathrm{F}_{\mu}[\partial_x r] + \frac{3\varepsilon}{2} r\partial_x r = \mu\varepsilon R_1, \\
            \partial_t s -\mathrm{F}_{\mu}[\partial_x s] - \frac{3\varepsilon}{2} s \partial_x s = \mu\varepsilon R_2.
        \end{cases}
    \end{align*}
\end{myrem}

\begin{mylem}\label{Solutions of the Whitham equations in W alpha,1}
    Let $\theta > 1/2$ and $\alpha \geq 0$. Then for any $r_0, s_0 \in W^{\alpha+5,1}(\mathbb{R})$ there exists a time $T >0$ such that both equations \eqref{Decoupled Whitham equations} admit a unique solution $r,s \in \mathcal{C}([0,\frac{T}{\varepsilon}], H^{\alpha+4}(\mathbb{R}))$, with initial conditions $r_0$ and $s_0$, which satisfy for any times $t \in [0,\frac{T}{\epsilon}]$
    \begin{align*}
        \begin{cases}
            |r(t)|_{W^{\alpha,1}} \lesssim (1 + \mu t)^{\theta} C(\mu_{\max},T,|r_0|_{W^{\alpha+5,1}}), \\
            |s(t)|_{W^{\alpha,1}} \lesssim (1 + \mu t)^{\theta} C(\mu_{\max},T,|s_0|_{W^{\alpha+5,1}}).
        \end{cases}
    \end{align*}
    In particular $r,s \in \mathcal{C}([0,\frac{T}{\varepsilon}], W^{\alpha,1}(\mathbb{R}))$ are uniformly bounded in $(\mu,\varepsilon)$ for times $t \in [0,\frac{T}{\max{(\mu,\varepsilon)}}]$ in $W^{\alpha,1}(\mathbb{R})$.  
    
    Moreover, for any $r_0, s_0 \in W^{\alpha+7,1}(\mathbb{R})$, both equations \eqref{Decoupled Whitham equations} admit a unique solution $r,s \in \mathcal{C}^1([0,\frac{T}{\max{(\mu,\varepsilon)}}],W^{\alpha,1}(\mathbb{R})) \cap \mathcal{C}([0,\frac{T}{\max{(\mu,\varepsilon)}}], W^{\alpha+2,1}(\mathbb{R}))$ with initial conditions $r_0$ and $s_0$. 
\end{mylem}
\begin{proof}
    We have the Sobolev embeddings $W^{\alpha+5,1}(\mathbb{R}) \subset H^{\alpha+4}(\mathbb{R})$ and the equations \eqref{Decoupled Whitham equations} are well-posed in these Sobolev spaces (see \cite{EhrestromEscherPei2015}). So there exists $T >0$ such that the latter equations admit a unique solution in $\mathcal{C}([0,\frac{T}{\varepsilon}],H^{\alpha+2}(\mathbb{R}))$. 
    
    We do the rest of the reasoning for $r$. We use the Duhamel formula to say that
    \begin{align*}
        r(t) = e^{-itD\mathrm{F}_{\mu}}r_0 + \frac{3\varepsilon}{2} \int_{0}^t e^{-i(t-t')D\mathrm{F}_{\mu}} \mathrm{F}_{\mu}[r \partial_x r] \dt'.
    \end{align*}
    We prove first that $|e^{-itD\mathrm{F}_{\mu}} r_0|_{W^{\alpha,1}} \lesssim (1+\mu t)^{\theta} |r_0|_{W^{\alpha+2,1}}$.
    
    We begin by doing the same reasoning as in the proof of Property \ref{Boundedness of the Poisson tensor}. We remark that we only need to prove the previous estimate in $L^1(\mathbb{R})$. Let $\chi \in \mathcal{C}^{\infty}((-2,2))$ be a cutoff function such that $\chi(\xi) = 1$ for $|\xi|\leq 1$. We decompose $e^{-itD\mathrm{F}_{\mu}} r_0 = e^{-it D (\mathrm{F}_{\mu}-1)} e^{-itD}r_0$ into
    \begin{align*}
        e^{-itD(\mathrm{F}_{\mu}-1)} e^{-itD}r_0 (x) &= \mathcal{F}^{-1}[e^{-it\xi (F_{\mu}(\xi)-1)}\chi(\xi)]*(e^{-itD}r_0) \\
        &+ \mathcal{F}^{-1}[e^{_it\xi (F_{\mu}(\xi)-1)}\frac{1-\chi(\xi)}{(i\xi)^2}]*\partial_y^2 e^{-itD} r_0  \colonequals I_1 + I_2.
    \end{align*}
    \begin{itemize}
        \item We prove here that $|I_1|_{L^1} \lesssim (1+\mu t)^{\frac{1}{2}+} |r_0|_{L^1}$.
        Using Young's inequality, we get
        \begin{align*}
            |I_1|_{L^1} \leq |\mathcal{F}^{-1}[e^{-it\xi (F_{\mu}(\xi)-1)}\chi(\xi)]|_{L^1} |e^{-itD}r_0|_{L^1} \leq |\mathcal{F}^{-1}[e^{-it\xi (F_{\mu}(\xi)-1)}\chi(\xi)]|_{L^1} |r_0|_{L^1}.
        \end{align*}
        Now we pose $G(x) = \mathcal{F}^{-1}[e^{-it\xi (F_{\mu}(\xi)-1)}\chi(\xi)]$. Let $\beta \in \mathcal{C}^{\infty}((-2,2))$ be a cutoff function such that $\beta(x) = 1$ for $|x| \leq 1$. We have for any $\theta > 1/2$ 
        \begin{align*}
            |G|_{L^1} &\leq |\beta G|_{L^1} + |(1-\beta) G|_{L^1} \\
            &\lesssim 1 + |\frac{1-\beta(x)}{x^{\theta}}|_{L^2}|x^{\theta} G|_{L^2}.
        \end{align*}
        We will prove the inequality $|xG|_{L^2} \lesssim 1 + \mu t$ and interpolate it with the obvious one $|G|_{L^2} \lesssim 1$. 
        
        We remark that 
        \begin{align*}
            |xG|_{L^2} = |\partial_{\xi} \widehat{G}|_{L^2} = |\partial_{\xi}(e^{-it\xi(F_{\mu}(\xi)-1)} \chi(\xi))|_{L^2}.
        \end{align*}
        But we have the two following estimates on $F_{\mu}$: for any $\xi \in \mathbb{R}$
        \begin{align}\label{estimates on Fmu}
                |F_{\mu}(\xi)-1| \lesssim \mu \xi^2, \ \ \ \ 
                |F_{\mu}'(\xi)| \lesssim \mu \xi^2.
        \end{align}
        So 
        \begin{align*}
            |\partial_{\xi}(e^{-it\xi(F_{\mu}(\xi)-1)} \chi(\xi))|_{L^2} &\leq t|(F_{\mu}(\xi)-1) \chi(\xi)|_{L^2} + t|\xi F_{\mu}'(\xi) \chi(\xi)|_{L^2} + |\chi'(\xi)|_{L^2} \lesssim 1+\mu t.
        \end{align*}
        Then Hölder's inequality gives us
        \begin{align*}
            |x^{\theta}G|_{L^2} = |(x G)^{\theta} G^{1-\theta}|_{L^2} \leq |(xG)^{\theta}|_{L^p}|G^{1-\theta}|_{L^q}, 
        \end{align*}
        with $\frac{1}{2} = \frac{1}{p} + \frac{1}{q}$. It remains to take $p= \frac{2}{\theta}$ and $q=\frac{2}{1-\theta}$ to get
        \begin{align*}
            |x^{\theta}G|_{L^2} \leq |x G|_{L^2}^{\theta}|G|_{L^2}^{1-\theta} \lesssim (1+\mu t)^{\theta}.
        \end{align*}
        \item We prove here that $|I_2|_{L^1} \lesssim (1+\mu t)^{\theta}|r_0|_{W^{2,1}}$. Again we use Young's inequality to get 
        \begin{align*}
            |I_2|_{L^1} \leq |\mathcal{F}^{-1}[e^{_it\xi (F_{\mu}(\xi)-1)}\frac{1-\chi(\xi)}{(i\xi)^2}]|_{L^1}|r_0|_{W^{2,1}}.
        \end{align*}
        Then the rest of the proof is the same, with a little exception: here, using the estimates \eqref{estimates on Fmu}, we have
        \begin{align*}
            &|\partial_{\xi}(e^{-it\xi (F_{\mu}(\xi)-1)}\frac{1-\chi(\xi)}{\xi^2})|_{L^2} \\
            \leq &|\frac{t(F_{\mu}(\xi)-1)(1-\chi(\xi))\xi^2 + t\xi F_{\mu}'(\xi)(1-\chi(\xi))\xi^2-\chi'(\xi)\xi^2 - 2(1-\chi(\xi))\xi}{\xi^4}|_{L^2}\\
            \lesssim& 1 + \mu t.
        \end{align*}
    \end{itemize}
    Now we prove $|\int_0^t e^{-i(t-t')D\mathrm{F}_{\mu}}\mathrm{F}_{\mu}[r\partial_x r] \dt'|_{W^{\alpha,1}} \lesssim (1+\mu t)^{\theta} \int_0^t |r|^2_{H^{\alpha+4}} \dt'$, so that for all times $t \in [0,\frac{T}{\varepsilon}]$ one has
    \begin{align*}
        |\frac{3\varepsilon}{2} \int_0^t e^{-i(t-t')D\mathrm{F}_{\mu}}\mathrm{F}_{\mu}[r\partial_x r] \dt'|_{W^{\alpha,1}} &\leq (1+\mu t)^{\theta}C(\mu_{\max},T,|r_0|_{H^{\alpha+4}}) \\
        &\leq (1 + \mu t)^{\theta} C(\mu_{\max},T,|r_0|_{W^{\alpha+5,1}}).
    \end{align*} 
    Using what we did before, we have
    \begin{align*}
        |\int_0^t e^{-i(t-t')D\mathrm{F}_{\mu}}\mathrm{F}_{\mu}[r\partial_x r] \dt'|_{W^{\alpha,1}} & \leq \int_0^t |e^{-i(t-t')D\mathrm{F}_{\mu}}\mathrm{F}_{\mu}[r\partial_x r]|_{W^{\alpha,1}} \dt' \\
        & \lesssim \int_0^t (1 + \mu (t-t'))^{\theta} |\mathrm{F}_{\mu}[r\partial_x r]|_{W^{\alpha+2,1}} \dt'.
    \end{align*}
    Using the continuity of $\mathrm{F}_{\mu}$ from $W^{\alpha+3,1}(\mathbb{R})$ to $W^{\alpha+2,1}(\mathbb{R})$ (see Property \ref{Boundedness of the Poisson tensor}), we get
    \begin{align*}
        |\int_0^t e^{-i(t-t')D\mathrm{F}_{\mu}}\mathrm{F}_{\mu}[r\partial_x r] \dt'|_{W^{\alpha,1}} &\lesssim (1+\mu t)^{\theta} \int_0^t |r\partial_x r|_{W^{\alpha+3,1}} \dt' \\
        &\lesssim (1+\mu t)^{\theta} \int_0^t |r|_{H^{\alpha+4}}^2 \dt'.
    \end{align*}
    It remains to prove that taking $r_0$ and $s_0$ in $W^{\alpha+7}(\mathbb{R})$, $\partial_t r$ and $\partial_t s$ are in $W^{\alpha,1}(\mathbb{R})$. We do the reasoning for $r$. From what we proved above, $r$ is in $W^{\alpha+2,1}(\mathbb{R})$. It is also solution of the equation
    \begin{align*}
        \partial_t r + \mathrm{F}_{\mu} [\partial_x r] + \frac{3\varepsilon}{2}\mathrm{F}_{\mu}[r\partial_x r] = 0.
    \end{align*}
    So that
    \begin{align*}
        |\partial_t r|_{W^{\alpha,1}} \lesssim &|\mathrm{F}_{\mu} [\partial_x r]|_{W^{\alpha,1}} + |\mathrm{F}_{\mu}[r\partial_x r]|_{W^{\alpha,1}} \\
        \leq &|r|_{W^{\alpha+2,1}} + |r\partial_x r|_{W^{\alpha+1,1}} \lesssim |r|^2_{W^{\alpha+2,1}}. 
    \end{align*}
\end{proof}

\begin{mythm}\label{From HWh to HBW}
    Let $\alpha \geq 0$. Let $r,s \in \mathcal{C}^1([0,\frac{T}{\varepsilon}], W^{\alpha+6,1}(\mathbb{R}))$ be solutions of the equations \eqref{Decoupled Whitham equations}. Let also $\mathcal{T}_{\rm{B}}$ be the transformation \eqref{explicit formulation of T}, and define:
    \begin{align}\label{rh and sh}
        \begin{pmatrix} r_{\rm{c}} \\ s_{\rm{c}} \end{pmatrix} \colonequals \mathcal{T}_{\rm{B}}(r,s).
    \end{align}
    Then, there exists $R \in \mathcal{C}^1([0,\frac{T}{\varepsilon}], H^{\alpha}(\mathbb{R})\times H^{\alpha})(\mathbb{R})$, such that one has:
    \begin{align*}
        \partial_t \begin{pmatrix} r_{\rm{c}} \\ s_{\rm{c}} \end{pmatrix} = J_{\mu} \nabla (H_{\BW}) (r_{\rm{c}},s_{\rm{c}}) + (\mu\varepsilon + \varepsilon^2) R, \ \ \forall t \in [0,\frac{T}{\varepsilon}], 
    \end{align*}
    where $H_{\BW}$ is the Hamiltonian defined in \eqref{Hamiltonian in r and s} and $|R|_{H^{\alpha}\times H^{\alpha}} \leq C(\mu_{\max},|r|_{W^{\alpha+6,1}},|s|_{W^{\alpha+6,1}})$. In particular, $R$ is uniformly bounded in $(\mu,\varepsilon)$ for times $t \in [0,\frac{T}{\max{(\mu,\varepsilon)}}]$ in $H^{\alpha}(\mathbb{R})\times H^{\alpha}(\mathbb{R})$.
\end{mythm}
\begin{proof}
    The existence of the solutions $(r,s)$ is given by the previous lemma \ref{Solutions of the Whitham equations in W alpha,1}.
    
    We write $\mathcal{T}_{\rm{B}}(r,s) = \begin{pmatrix} T_1(r,s) \\ T_2(r,s) \end{pmatrix} = \begin{pmatrix}  r + \frac{\varepsilon}{4} \partial_x r \partial^{-1}s + \frac{\varepsilon}{4} rs + \frac{\varepsilon}{8}s^2 \\ s + \frac{\varepsilon}{4} \partial_x s \partial^{-1}r + \frac{\varepsilon}{4} rs + \frac{\varepsilon}{8}r^2 \end{pmatrix}$. By definition of $H_{\BW}$ (see \eqref{Hamiltonian in r and s}) and $H_{\Wh}$ (see Notation \ref{Notation Hamiltonian HWh}), and the computations in the proof of Proposition \ref{proposition explicit formulation of T} we have
    \begin{align*}
        H_{\BW}(\mathcal{T}_{\rm{B}}(r,s)) &= \int_{\mathbb{R}} (T_1(r,s)^2 + T_2(r,s)^2) \dx + \frac{\varepsilon}{2} \int (T_1(r,s)^3 + T_2(r,s)^3) \dx \\
        &-\frac{\varepsilon}{2} \int_{\mathbb{R}} (T_1(r,s)^2 T_2(r,s) + T_1(r,s) T_2(r,s)^2) \dx \\
        &= H_{\Wh}(r,s) \\
        &+ \int_{\mathbb{R}} [(\frac{\varepsilon}{4} \partial_x (r) \partial^{-1}(s) + \frac{\varepsilon}{4} rs + \frac{\varepsilon}{8}s^2)^2 + (\frac{\varepsilon}{4} \partial_x (s) \partial^{-1}(r) + \frac{\varepsilon}{4} rs + \frac{\varepsilon}{8}r^2)^2] \dx \\
        &+ \frac{\varepsilon}{2} \int_{\mathbb{R}} (T_1(r,s)^3-r^3 + T_2(r,s)^3-s^3) \dx \\
        &- \frac{\varepsilon}{2} \int_{\mathbb{R}} (T_1(r,s)^2T_2(r,s) - r^2s + T_1(r,s)T_2(r,s)^2 - rs^2) \dx 
    \end{align*}
    So $H_{\BW}(\mathcal{T}_{\rm{B}}(r,s)) = H_{\Wh}(r,s) + \varepsilon^2 \int_{\mathbb{R}} p(r,s,\partial_x (r) \partial^{-1}(s), \partial_x (s)\partial^{-1}(r)) \dx$, where $p:\mathbb{R}^4 \to \mathbb{R}$ is polynomial.
    
    But $(r_{\rm{c}},s_{\rm{c}}) = \mathcal{T}_{\rm{B}}(r,s)$. So
    \begin{align*}
        \partial_t \begin{pmatrix} r_{\rm{c}} \\ s_{\rm{c}}\end{pmatrix} = \frac{\partial \mathcal{T}_{\rm{B}}}{\partial(r,s)} \partial_t \begin{pmatrix} r \\ s \end{pmatrix}, 
    \end{align*}
    where $\frac{\partial \mathcal{T}_{\rm{B}}}{\partial(r,s)}$ is the Jacobian matrix of $\mathcal{T}_{\rm{B}}(r,s)$ computed in the proof of Proposition \ref{preservation of the Hamiltonian structure}. Since $r$ and $s$ are solutions of the Hamilton's equations associated with $H_{\Wh}$, we have
    \begin{align*}
        \partial_t \begin{pmatrix} r_{\rm{c}} \\ s_{\rm{c}}\end{pmatrix} &= \frac{\partial \mathcal{T}_{\rm{B}}}{\partial(r,s)} J_{\mu} \nabla (H_{\Wh})(r,s) \\
        &= \frac{\partial \mathcal{T}_{\rm{B}}}{\partial(r,s)} J_{\mu} \nabla(H_{\BW}(\mathcal{T}_{\rm{B}}(r,s)) - \varepsilon^2 \frac{\partial \mathcal{T}_{\rm{B}}}{\partial(r,s)} J_{\mu} \nabla \int_{\mathbb{R}} p(r,s,\partial_x r \partial^{-1}s, \partial_x s\partial^{-1}r) \dx \\
        &= \frac{\partial \mathcal{T}_{\rm{B}}}{\partial(r,s)} J_{\mu} \Big(\frac{\partial \mathcal{T}_{\rm{B}}}{\partial(r,s)}\Big)^* \nabla(H_{\BW})(r_{\rm{c}},s_{\rm{c}}) - \varepsilon^2 \frac{\partial \mathcal{T}_{\rm{B}}}{\partial(r,s)} J_{\mu} \nabla \int_{\mathbb{R}} p(r,s,\partial_x r \partial^{-1}s, \partial_x s\partial^{-1}r) \dx.
    \end{align*}
    Moreover by Proposition \ref{preservation of the Hamiltonian structure}, we know that $\frac{\partial \mathcal{T}_{\rm{B}}}{\partial(r,s)} J_{\mu} \Big(\frac{\partial \mathcal{T}_{\rm{B}}}{\partial(r,s)}\Big)^* = J_{\mu} + (\mu\varepsilon + \varepsilon^2)R$, with $|RU|_{H^{\alpha}\times H^{\alpha}} \leq C(\mu_{\max}, |r|_{W^{\alpha+4,1}},|s|_{W^{\alpha+4,1}})|U|_{H^{\alpha+4}}$ for any $U \in H^{\alpha+4}(\mathbb{R})\times H^{\alpha+4}(\mathbb{R})$. So, here, we need to estimate $\nabla(H_{\BW})(r_{\rm{c}},s_{\rm{c}})$ in $H^{\alpha+4}(\mathbb{R})\times H^{\alpha+4}(\mathbb{R})$.
    \begin{align*}
        \nabla(H_{\BW})(\mathcal{T}_{\rm{B}}(r,s)) = \begin{pmatrix} 2 T_1(r,s) + \frac{3\varepsilon}{2} T_1(r,s)^2 - \varepsilon T_1(r,s)T_2(r,s) - \frac{\varepsilon}{2} T_2(r,s)^2 \\ 2 T_2(r,s) + \frac{3\varepsilon}{2} T_2(r,s)^2 - \varepsilon T_1(r,s)T_2(r,s) - \frac{\varepsilon}{2} T_1(r,s)^2 \end{pmatrix}.
    \end{align*}
    Combining the algebra properties of $H^{\alpha+4}(\mathbb{R})$ with the estimates 
    \begin{align*}
    \begin{cases}
        |\partial_x (r) \partial^{-1}(s)|_{H^{\alpha+4}}\lesssim |\partial^{-1}s|_{W^{\alpha+4,\infty}}|\partial_x r|_{H^{\alpha+4}} \lesssim  |s|_{W^{\alpha+4,1}}|r|_{H^{\alpha+5}} \lesssim |s|_{W^{\alpha+4,1}}|r|_{W^{\alpha+6,1}}, \\
        |\partial_x (s)  \partial^{-1}(r)|_{H^{\alpha+4}}\lesssim |r|_{W^{\alpha+4,1}}|s|_{W^{\alpha+6,1}}
    \end{cases}
    \end{align*}
    (see \eqref{estimates in Sobolev infinity}), we have 
    \begin{align*}
        |\nabla(H_{\BW})(\mathcal{T}_{\rm{B}}(r,s))|_{H^{\alpha+4}} \leq C(||r|_{W^{\alpha+6,1}},|s|_{W^{\alpha+6,1}}).
    \end{align*}
    It remains to prove that  $\frac{\partial \mathcal{T}_{\rm{B}}}{\partial(r,s)} J_{\mu} \nabla \int_{\mathbb{R}} p(r,s,\partial_x (r) \partial^{-1}(s), \partial_x (s)\partial^{-1}(r)) \dx$ is in $\mathcal{C}^1([0,\frac{T}{\varepsilon}],H^{\alpha}(\mathbb{R})\times H^{\alpha}(\mathbb{R}))$.
    
    The polynomial $p$ is a linear combination of terms of the form $r^{n_1}s^{n_2}(\partial_x (r) \partial^{-1}(s))^{n_3} (\partial_x (s) \partial^{-1}(r))^{n_4}$ where $n_1, n_2, n_3$ and $n_4$ are non-negative integers such that $n_1 + n_2 + n_3 + n_4 \geq 2$. Moreover
    \begin{align*}
        J_{\mu} \nabla \int_{\mathbb{R}} r^{n_1}s^{n_2}(\partial_x (r) \partial^{-1}(s))^{n_3} (\partial_x (s) \partial^{-1}(r))^{n_4} \dx = \begin{pmatrix} M_1 \\ M_2 \end{pmatrix},
    \end{align*}
    where
    \begin{align*}
        \begin{cases}
            M_1 = &-\frac{1}{2}\mathrm{F}_{\mu} \Big[\partial_x \Big( n_1 r^{n_1-1} s^{n_2}(\partial_x (r) \partial^{-1}(s))^{n_3} (\partial_x (s) \partial^{-1}(r))^{n_4}\Big) \Big] \\
            &+ \frac{1}{2} \mathrm{F}_{\mu}\Big[\partial_x^2 \Big( r^{n_1} s^{n_2} n_3(\partial_x r)^{n_3-1}(\partial^{-1}s)^{n_3} (\partial_x (s) \partial^{-1}(r))^{n_4} \Big)\Big]\\
            &+ \frac{1}{2} \mathrm{F}_{\mu}\Big[ r^{n_1} s^{n_2} (\partial_x (r) \partial^{-1} (s))^{n_3} (\partial_x s)^{n_4} n_4 (\partial^{-1} r)^{n_4-1} \Big], \\
            M_2 = &\frac{1}{2}\mathrm{F}_{\mu}\Big[\partial_x\Big(r^{n_1} n_2 s^{n_2-1}(\partial_x (r) \partial^{-1}(s))^{n_3} (\partial_x (s) \partial^{-1}(r))^{n_4}\Big)\Big] \\
            &- \frac{1}{2}\mathrm{F}_{\mu}\Big[r^{n_1} s^{n_2} (\partial_x r)^{n_3}n_3(\partial^{-1}s)^{n_3-1} (\partial_x (s) \partial^{-1}(r))^{n_4} \Big] \\
            &-\frac{1}{2} \mathrm{F}_{\mu}\Big[\partial_x^2 \Big( r^{n_1} s^{n_2} (\partial_x (r) \partial^{-1} (s))^{n_3} n_4(\partial_x s)^{n_4-1}(\partial^{-1} r)^{n_4} \Big)\Big],
        \end{cases}
    \end{align*}
    where we used $\partial_x \partial^{-1}u = u$ when $u \in L^1(\mathbb{R})$ (see Definition/Property \ref{primitive operator}).
    The property $n_1 + n_2 + n_3 + n_4 \geq 2$ implies that the terms containing $\partial^{-1}s$ or $\partial^{-1}r$ are always multiplied by a function in a Sobolev space $W^{\beta,1}(\mathbb{R})$ with $\beta \geq 0$.
    
    We recall now the expression of $\frac{\partial \mathcal{T}_{\rm{B}}}{\partial(r,s)}$:
    \begin{align*}
        \frac{\partial \mathcal{T}_{\rm{B}}}{\partial(r,s)} = \begin{pmatrix} 1 + \frac{\varepsilon}{4} s + \frac{\varepsilon}{4} \partial^{-1}(s) \partial_x(\circ) & \frac{\varepsilon}{4} \partial_x (r) \partial^{-1}(\circ) + \frac{\varepsilon}{4}(r+s) \\ \frac{\varepsilon}{4} \partial_x (s) \partial^{-1}(\circ) + \frac{\varepsilon}{4}(r+s) & 1 + \frac{\varepsilon}{4} r + \frac{\varepsilon}{4} \partial^{-1}(r) \partial_x(\circ)\end{pmatrix}.
    \end{align*}
    From the above computations we get
    \begin{align}\label{Rest}
        &\frac{\partial \mathcal{T}_{\rm{B}}}{\partial(r,s)} J_{\mu} \nabla \int_{\mathbb{R}} p(r,s,\partial_x (r) \partial^{-1}(s), \partial_x (s)\partial^{-1}(r)) \dx \\
        = &\begin{pmatrix} M_1 + \frac{\varepsilon}{4} s M_1 + \frac{\varepsilon}{4} \partial^{-1}(s) \partial_x(M_1) + \frac{\varepsilon}{4} \partial_x (r) \partial^{-1}(M_2) + \frac{\varepsilon}{4}(r+s)M_2 \\ \frac{\varepsilon}{4} \partial_x (s) \partial^{-1}(M_1) + \frac{\varepsilon}{4}(r+s)M_1 + M_2 + \frac{\varepsilon}{4} r M_2 + \frac{\varepsilon}{4} \partial^{-1}(r) \partial_x(M_2) \end{pmatrix},
    \end{align}
    We estimate the terms thus obtained. Using product estimates \ref{product estimate} and \eqref{estimates in Sobolev infinity} we have
    \begin{align*}
        \begin{cases}
            |s M_1|_{H^{\alpha}} \lesssim |s|_{H^{\alpha+1}}|M_1|_{H^{\alpha}}, \\
            |\partial^{-1}(s)\partial_x(M_1)|_{H^{\alpha}} \lesssim |\partial^{-1}s|_{W^{\alpha,\infty}}|M_1|_{H^{\alpha+1}} \lesssim |s|_{W^{\alpha,1}}|M_1|_{H^{\alpha+1}}, \\
            |\partial_x(r)\partial^{-1}(M_1)|_{H^{\alpha}} \lesssim |r|_{H^{\alpha+1}}|\partial^{-1}M_1|_{W^{\alpha,\infty}} \lesssim  |r|_{W^{\alpha+2,1}}|M_1|_{W^{\alpha,1}}.
        \end{cases}
    \end{align*}
    Idem for the terms with $M_2$. So that it only remains to estimate $M_1$ and $M_2$. We do it for $M_1$. Using product estimates \ref{product estimate} and $n_1 + n_2 + n_3 +n_4 \geq 2$, we get
    \begin{align*}
        |M_1|_{H^{\alpha}} \lesssim |\partial^{-1}r|_{W^{\alpha+2,\infty}}^{N_1}|\partial^{-1}s|_{W^{\alpha+2,\infty}}^{N_2}|r|_{H^{\alpha+3}}^{N_3}|s|_{H^{\alpha+3}}^{N_4}, 
    \end{align*}  
    where $N_1, N_2, N_3$ and $N_4$ are non negative integers such that $(N_3, N_4) \neq (0,0)$. Then, using \eqref{estimates in Sobolev infinity} and Sobolev embeddings, we have
    \begin{align*}
        |M_1|_{H^{\alpha}} \lesssim |r|_{W^{\alpha+4,1}}^{N_1+N_3}|s|_{W^{\alpha+4,1}}^{N_2+N_4}.
    \end{align*}
    Moreover, using the continuity of $\mathrm{F}_{\mu}$ from $W^{\alpha+1,1}(\mathbb{R})$ to $W^{\alpha,1}(\mathbb{R})$ and the algbra properties of $W^{\beta,1}(\mathbb{R})$ for any $\beta \geq 1$, we get
    \begin{align*}
        |M_1|_{W^{\alpha,1}} \lesssim |\partial^{-1}r|_{W^{\alpha+3,\infty}}^{N_1}|\partial^{-1}s|_{W^{\alpha+3,\infty}}^{N_2}|r|_{W^{\alpha+4,1}}^{N_3}|s|_{W^{\alpha+4,1}}^{N_4} \lesssim |r|_{W^{\alpha+4,1}}^{N_1 + N_3}|s|_{W^{\alpha+4,1}}^{N_2+N_4}.
    \end{align*}
\end{proof}

Combining Proposition \ref{From HBX to HWW} and Theorem \ref{From HWh to HBW} we get Theorem \ref{Final theorem introduction} which we recall here for the sake of clarity. 
\begin{mythm}\label{theorem end of the article}
    Let $\alpha \geq 0$. Let $r,s \in \mathcal{C}^1([0,\frac{T}{\varepsilon}], W^{\alpha+7,1}(\mathbb{R}))$ be solutions of the equations \eqref{Decoupled Whitham equations}. Let also $\mathcal{T}_{\rm{I}}, \mathcal{T}_{\rm{D}}$ and $\mathcal{T}_{B}$ be the transformations of Definition \ref{Definition transformations} and define:
    \begin{align}\label{zeta and psi}
        \begin{pmatrix} \zeta_{\Wh} \\ \psi_{\Wh} \end{pmatrix} \colonequals \mathcal{T}_{\rm{I}}(\mathcal{T}_{\rm{D}}(\mathcal{T}_{\rm{B}}(r,s))).
    \end{align}
    Then, there exists $R \in \mathcal{C}^1([0,\frac{T}{\varepsilon}], H^{\alpha}(\mathbb{R})\times H^{\alpha}(\mathbb{R}))$, such that one has:
    \begin{align*}
        \partial_t \begin{pmatrix} \zeta_{\Wh} \\ \psi_{\Wh} \end{pmatrix} = J \nabla (H_{\WW}) (\zeta_{\Wh},\psi_{\Wh}) + (\mu\varepsilon + \varepsilon^2) R, \ \ \forall t \in [0,\frac{T}{\varepsilon}], 
    \end{align*}
    where $H_{\WW}$ is the Hamiltonian defined in \eqref{Hamiltonian Water Waves equations} and $|R|_{H^{\alpha}\times \dot{H}^{\alpha+1}} \leq C(\mu_{\max},|r|_{W^{\alpha+7,1}},|s|_{W^{\alpha+7,1}})$. In particular, $R$ is uniformly bounded in $(\mu,\varepsilon)$ for times $t \in [0,\frac{T}{\max{(\mu,\varepsilon)}}]$ in $H^{\alpha}(\mathbb{R})\times \dot{H}^{\alpha+1}(\mathbb{R})$.
\end{mythm}
\begin{proof}
    We start by defining the quantities 
    \begin{align}\label{zeta and partial psi}
        \begin{pmatrix} \zeta_{\Wh} \\ v_{\Wh} \end{pmatrix} \colonequals \mathcal{T}_{\rm{D}}(\mathcal{T}_{\rm{B}}(r,s)).
    \end{align}
    Differentiating \eqref{zeta and partial psi} in time and using Theorem \ref{From HWh to HBW} we get
    \begin{align*}
        \partial_t \begin{pmatrix} \zeta_{\Wh} \\ v_{\Wh} \end{pmatrix} = & \frac{\partial \mathcal{T}_{\rm{D}}}{\partial(r,s)}(\mathcal{T}_{\rm{B}}(r,s)) \frac{\partial \mathcal{T}_{\rm{B}}}{\partial(r,s)} \partial_t \begin{pmatrix} r \\ s \end{pmatrix} \\
        = &\frac{\partial \mathcal{T}_{\rm{D}}}{\partial(r,s)}(\mathcal{T}_{\rm{B}}(r,s)) J_{\mu} \nabla(H_{\BW})(\mathcal{T}_{\rm{B}}(r,s)) + (\mu\varepsilon + \varepsilon^2)\frac{\partial \mathcal{T}_{\rm{D}}}{\partial(r,s)}(\mathcal{T}_{\rm{B}}(r,s))R_1,
    \end{align*}
    where $\frac{\partial \mathcal{T}_{\rm{D}}}{\partial(r,s)}(\mathcal{T}_{\rm{B}}(r,s)) = \begin{pmatrix} 1 & 1 \\ \mathrm{F}_{\mu}^{-1}[\circ] & -\mathrm{F}_{\mu}^{-1}[\circ] \end{pmatrix}$ and $R_1$ is the rest in Theorem \ref{From HWh to HBW}. Using now Proposition \ref{From HBW to H0 epsilon H1} we get
    \begin{align*}
        \partial_t \begin{pmatrix} \zeta_{\Wh} \\ v_{\Wh} \end{pmatrix} &= \widetilde{J} \nabla(\widetilde{H}_0 + \varepsilon \widetilde{H}_1)(\mathcal{T}_{\rm{D}}(\mathcal{T}_{\rm{B}}(r,s)) + (\mu\varepsilon + \varepsilon^2)R_2 \\
        &= \widetilde{J} \nabla(\widetilde{H}_0 + \varepsilon \widetilde{H}_1)(\zeta_{\Wh},v_{\Wh}) + (\mu\varepsilon + \varepsilon^2)R_2,
    \end{align*}
    where $\widetilde{J} = \begin{pmatrix} 0 & -\partial_x \\ -\partial_x & 0 \end{pmatrix}$ and $|R_2|_{H^{\alpha}\times H^{\alpha}} \leq C(|r|_{W^{\alpha+7,1}},|s|_{W^{\alpha+7,1}})$ (because for any $\beta \geq 0$ and any $u \in W^{\beta+1,1}(\mathbb{R})$, $|\mathrm{F}_{\mu}^{-1}[u]|_{H^{\beta}} \lesssim |u|_{H^{\beta+1/2}} \lesssim |u|_{W^{\beta+1}}$). Changing the unknowns into $(\zeta_{\Wh},\psi_{\Wh})$ we obtain
    \begin{align*}
        \partial_t \begin{pmatrix} \zeta_{\Wh} \\ \psi_{\Wh} \end{pmatrix} = J \nabla(H_0 + \varepsilon H_1)(\zeta_{\Wh},\psi_{\Wh}) + (\mu\varepsilon + \varepsilon^2) R_3, 
    \end{align*}
    where $J = \begin{pmatrix} 0 & 1 \\ -1 & 0 \end{pmatrix}$ and $|R_3|_{H^{\alpha}\times \dot{H}^{\alpha+1}} \leq C(|r|_{W^{\alpha+7,1}},|s|_{W^{\alpha+7,1}})$  So that
    \begin{align*}
        \partial_t \begin{pmatrix} \zeta_{\Wh} \\ \psi_{\Wh} \end{pmatrix} &= J \nabla(H_{\WW})(\zeta_{\Wh}, \psi_{\Wh}) + (J\nabla(H_0 + \varepsilon H_1)(\zeta_{\Wh},\psi_{\Wh}) - J\nabla (H_{\WW})(\zeta_{\Wh},\psi_{\Wh}))\\
        &+ (\mu\varepsilon + \varepsilon^2)R_2.
    \end{align*}
    Using Proposition \ref{From HBX to HWW} we get the result. The uniform boundedness in $(\mu,\varepsilon)$ for times $t \in [0,\frac{T}{\max{(\mu,\varepsilon)}}]$ of the remainder comes from the lemma \ref{Solutions of the Whitham equations in W alpha,1}.
\end{proof}

It only remains to prove Corollary \ref{Final corollary}.
\begin{proof}
    We construct $(r_0,s_0)$ such that
    \begin{align}\label{error initial condition}
        |\mathcal{T}_{\rm{I}}(\mathcal{T}_{\rm{D}}(\mathcal{T}_{\rm{B}}(r_0,s_0)))-(\zeta_0,\psi_0)|_{H^{\alpha}\times \dot{H}^{\alpha+1}} \lesssim \varepsilon^2.
    \end{align}
     We remark first that the transformations $\mathcal{T}_{\rm{I}}$ and $\mathcal{T}_{\rm{D}}$ are invertible with, for any $\beta \geq 0$, $\mathcal{T}_{\rm{I}}^{-1}:W^{\beta,1}(\mathbb{R})\times \dot{W}^{\beta+1,1}(\mathbb{R}) \to W^{\beta,1}(\mathbb{R})\times W^{\beta,1}(\mathbb{R})$ and $\mathcal{T}_{\rm{D}}^{-1}:W^{\beta+1,1}(\mathbb{R})\times W^{\beta+1,1}(\mathbb{R}) \to W^{\beta,1}(\mathbb{R})\times W^{\beta,1}(\mathbb{R})$ due to the continuity of $\mathrm{F}_{\mu}$ from $W^{\beta+1,1}(\mathbb{R})$ to $W^{\beta,1}(\mathbb{R})$. Moreover, by definition (see \eqref{explicit formulation of T}), we can write $\mathcal{T}_{\rm{B}}$ under the form
    \begin{align*}
        \mathcal{T}_{\rm{B}} = \rm{I}d + \varepsilon \widetilde{\mathcal{T}_{\rm{B}}},  
    \end{align*}
    where $\rm{I}d$ is the identity and for any $\beta \geq 0$, $\widetilde{\mathcal{T}_{\rm{B}}}: W^{\beta+1,1}(\mathbb{R})\times W^{\beta+1,1}(\mathbb{R}) \to W^{\beta,1}(\mathbb{R})\times W^{\beta,1}(\mathbb{R})$. We define the transformation $\mathcal{T}_{\rm{B}}^{\inv}: W^{\beta+1,1}(\mathbb{R})\times W^{\beta+1,1}(\mathbb{R}) \to W^{\beta,1}(\mathbb{R})\times W^{\beta,1}(\mathbb{R})$ by
    \begin{align*}
        \mathcal{T}_{\rm{B}}^{\inv} = \rm{I}d - \varepsilon \widetilde{\mathcal{T}_{\rm{B}}}.
    \end{align*}
    One has for any $(\eta, w) \in W^{\beta+2,1}(\mathbb{R})\times W^{\beta+2,1}(\mathbb{R})$, $\mathcal{T}_{\rm{B}} ( \mathcal{T}_{\rm{B}}^{\inv}(\eta,w)) = \begin{pmatrix} \eta \\ w \end{pmatrix}+ \varepsilon^2 R$ where $|R|_{W^{\beta,1}\times W^{\beta,1}} \leq C(|\eta|_{W^{\beta+2,1}},|w|_{W^{\beta+2,1}})$. So that 
    \begin{align}\label{r0 and s0}
        (r_0,s_0) = \mathcal{T}_{\rm{B}}^{\inv}(\mathcal{T}_{\rm{D}}^{-1}(\mathcal{T}_{\rm{I}}^{-1}(\zeta_0,\psi_0)))
    \end{align}
    satisfy \eqref{error initial condition}.
\end{proof}

\label{From the Hamiltonian of the Whitham-Boussinesq system under normal form to two decoupled Whitham equations}

\appendix
\section{Technical tools}
\begin{mypp}\label{product estimate}(Product estimates)
    \begin{enumerate}
        \item Let $t_0 > 1/2$, $s\geq-t_0$ and $f \in H^{\alpha}\cap H^{t_0}(\mathbb{R}), g\in H^{\alpha}(\mathbb{R})$. Then $fg \in H^{\alpha}(\mathbb{R})$ and 
        \begin{align*}
            |fg|_{H^{\alpha}} \lesssim |f|_{H^{\max{(t_0,\alpha)}}}|g|_{H^{\alpha}}
        \end{align*}
        \item Let $\alpha_1, \alpha_2 \in \mathbb{R}$ be such that $\alpha_1+\alpha_2 \geq 0$. Then for all $\alpha\leq \alpha_j$ $(j = 1,2)$ and $\alpha < \alpha_1 + \alpha_2 - 1/2$, and all $f\in H^{\alpha_1}(\mathbb{R}), g \in H^{\alpha_2}(\mathbb{R})$, one has $fg \in H^{\alpha}(\mathbb{R})$ and 
        \begin{align*}
            |fg|_{H^{\alpha}} \lesssim |f|_{H^{\alpha_1}}|g|_{H^{\alpha_2}}
        \end{align*}
    \end{enumerate}
\end{mypp}
\begin{proof}
See Appendix B.1 in \cite{WWP}.
\end{proof}

\begin{mypp}\label{Composition estimate}(Composition estimates)
    Let $G : \mathbb{R} \to \mathbb{R}$ be a smooth function such that $G(0)=0$. Also, let $t_0 > 1/2$, $\alpha \geq 0$ and $f \in H^{\max{(t_0,\alpha)}}(\mathbb{R})$. Then $G(f) \in H^{\alpha}(\mathbb{R})$ and 
    \begin{align*}
        |G(f)|_{H^{\alpha}} \leq C(|f|_{H^{\max{(t_0,\alpha)}}})|f|_{H^{\alpha}}.
    \end{align*}
\end{mypp}
\begin{proof}
    See Appendix B.1 in \cite{WWP}.
\end{proof}

\begin{mypp}\label{Quotient estimate}(Quotient estimates)
    Let $t_0>1/2, \alpha \geq-t_0$ and $c_0>0$. Also let $f \in H^{\alpha}(\mathbb{R})$ and $g\in H^{\alpha}\cap H^{t_0}(\mathbb{R})$ be such that for all $x \in \mathbb{R}$, one has $1 + g(X) \geq c_0$. Then $\frac{f}{1+g}$ belongs to $H^{\alpha}(\mathbb{R})$ and 
    \begin{align*}
        |\frac{f}{1+g}|_{H^{\alpha}} \leq C(\frac{1}{c_0},|g|_{H^{\max{(t_0,\alpha)}}})|f|_{H^{\alpha}}
    \end{align*}
\end{mypp}
\begin{proof}
    See Appendix B.1 in \cite{WWP}.
\end{proof}

\begin{mydef}\label{definition order of a Fourier multiplier}
We say that a Fourier multiplier $\mathrm{F}(D)$ is of order $\alpha $ $(\alpha \in\mathbb{R})$ and write $\mathrm{F} \in \mathcal{S}^\alpha $ if $\xi \in \mathbb{R} \mapsto F(\xi) \in \mathbb{C}$ is smooth and satisfies
\begin{align*}
    \forall \xi \in \mathbb{R}, \forall\beta\in\mathbb{N}, \ \ \sup_{\xi\in\mathbb{R}} \langle\xi\rangle^{\beta-\alpha }|\partial^{\beta}F(\xi)| < \infty.
\end{align*}
We also introduce the seminorm
\begin{align*}
    \mathcal{N}^\alpha (F) = \sup_{\beta\in\mathbb{N},\beta\leq 4} \sup_{\xi\in\mathbb{R}} \langle\xi\rangle^{\beta-\alpha }|\partial^{\beta}F(\xi)|.
\end{align*}
\end{mydef}

\begin{mypp}\label{Commutator estimates 2.0}
    Let $t_0 > 1/2$, $\alpha  \geq 0$ and $F \in \mathcal{S}^\alpha $. If $f\in H^\alpha \cap H^{t_0+1}$ then, for all $g \in H^{\alpha -1}$,
    \begin{align*}
        |[F(D),f]g|_2 \leq \mathcal{N}^\alpha (F)|f|_{H^{\max{(t_0+1,\alpha )}}}|g|_{H^{\alpha -1}}. 
    \end{align*}
\end{mypp}
\begin{proof}
    See Appendix B.2 in \cite{WWP} for a proof of this proposition.
\end{proof}

\begin{mypp}\label{Theorem 3.15}
    Let $\alpha\geq 2$. Let $\zeta \in H^{\alpha+2}(\mathbb{R})$ be such that \eqref{Non-Cavitation Hypothesis} is satisfied and $\psi \in \dot{H}^{\alpha+2}(\mathbb{R})$.
    Then one has 
    \begin{align*}
            |\frac{1}{\mu}\mathcal{G}^{\mu}[\varepsilon\zeta]\psi|_{H^{\alpha}} \leq M(s+2)|\psi|_{\dot{H}^{\alpha+2}}.
    \end{align*}
\end{mypp}
\begin{proof}
    This is a direct consequence of Theorem 3.15 in \cite{WWP}.
\end{proof}

\bibliographystyle{plain}
\bibliography{main.bbl}

\end{document}